\documentclass[11pt]{article}
\usepackage{latexsym}
\usepackage{amsmath}
\usepackage{amsthm,color}
\usepackage{amssymb}
\usepackage{mathrsfs}
\usepackage[colorlinks,  linkcolor=blue,  anchorcolor=blue, citecolor=blue]{hyperref}

\usepackage{lscape}
\newtheorem{theorem}{\indent Theorem}[section]

\newtheorem{lemma}[theorem]{\indent Lemma}
\newtheorem{remark}{\indent Remark}[section]

\begin{document}
\renewcommand{\baselinestretch}{1.3}


\begin{center}
    {\large \bf Normalized solutions for a coupled fractional Schr\"{o}dinger system in low dimensions}
\vspace{0.5cm}\\{\sc Meng Li$^1$}, {\sc Jinchun He$^1$}, {\sc Haoyuan Xu$^1$} and {\sc Meihua  Yang$^1$$^*$}\\
{\small 1. School of Mathematics and Statistics, Huazhong University of Science and Technology} \\
{\small Wuhan, 430074, China}
\end{center}


\renewcommand{\theequation}{\arabic{section}.\arabic{equation}}
\numberwithin{equation}{section}


\begin{abstract}
We consider the following coupled fractional Schr\"{o}dinger system:
\begin{equation*}
\left\{
\begin{aligned}
&(-\Delta)^su+\lambda_1u=\mu_1|u|^{2p-2}u+\beta|v|^p|u|^{p-2}u\\
&(-\Delta)^sv+\lambda_2v=\mu_2|v|^{2p-2}v+\beta|u|^p|v|^{p-2}v\\
\end{aligned}
\right.\quad\text{in}~{\mathbb{R}^N},
\end{equation*}
with $0<s<1$, $2s<N\le 4s$ and $1+\frac{2s}{N}<p<\frac{N}{N-2s}$, under the following constraint
\begin{align*}
\int_{\mathbb{R}^N}|u|^2dx=a_1^2\quad\text{and}\quad \int_{\mathbb{R}^N}|v|^2dx=a_2^2.
\end{align*}
Assuming that the parameters $\mu_1,\mu_2,a_1, a_2$ are fixed quantities, we prove the existence of normalized solution for different ranges of the coupling parameter $\beta>0$ .

\end{abstract}

\vspace{-1 cm}

\footnote[0]{ \hspace*{-7.4mm}
$^{*}$ Corresponding author.
E-mails:  yangmeih@hust.edu.cn}

\section{Introduction}

In this paper, we consider the following fractional Schr\"{o}dinger system with $1+\frac{2s}{N}<p<\frac{N}{N-2s}~\text{and}~2s<N\le 4s$,
\begin{equation}\label{GPE}
\left\{
\begin{aligned}
&(-\Delta)^su+\lambda_1u=\mu_1|u|^{2p-2}u+\beta|v|^p|u|^{p-2}u\\
&(-\Delta)^sv+\lambda_2v=\mu_2|v|^{2p-2}v+\beta|u|^p|v|^{p-2}v\\
\end{aligned}
\right.\quad\text{in}~{\mathbb{R}^N}.
\end{equation}
Here,
\begin{align}\label{normal}
\int_{\mathbb{R}^N}|u|^2dx=a_1^2\quad\text{and}\quad\int_{\mathbb{R}^N}|v|^2dx=a_2^2.
\end{align}

The parameters $\mu_1,~\mu_2$ and $\beta$ can be positive or negative. In the case of Laplacian, for $\mu_1,~\mu_2$ and $\beta$ are positive (resp. negative), the system is attractive (resp. repulsive).

Moreover, we analyze the existence of normalized solution of the system \eqref{GPE} for the case that the intraspecies interaction and the interspecies interaction are both attractive, i.e. $\mu_1>0,~\mu_2>0~\text{and}~\beta>0$.

One refers to this type of solutions as normalized solutions, since \eqref{normal} imposes a normalization on the $L^2$-masses of $u$ and $v$. This fact implies that $\lambda_1$ and $\lambda_2$ cannot be determined a priori, but are part of the
unknown.

The normalized solution of nonlinear
Schr\"{o}dinger equations and systems has gradually attracted the attention of a large number of researchers in recent years, both for the pure mathematical research and in view of  its very important applications in many physical problems, see more detail \cite{BC13,BJ18,BJN18,LYZ18}.

On one hand, for the nonlinear
Schr\"{o}dinger equation with $s=1$, in \cite{S20}, the author studied  existence and properties of ground states for the nonlinear Schr\"{o}dinger equation with combined power nonlinearities $p,q$ which satisfy $2<q\leq 2+\frac{4}{N}\leq p, p\neq q$. In \cite{SN20},  the author studied  existence and properties of ground states for the nonlinear Schr\"{o}dinger equation with combined power nonlinearities $q, 2^*$.

On the other hand, for the nonlinear
Schr\"{o}dinger system with $s=1$, Thomas et al. \cite{BJN18} recently proved the existence of positive solutions for the system with any arbitrary number of components in three-dimensional space.  In \cite{GJ18}, the authors considered the existence of multiple positive solutions to the nonlinear Schr\"{o}dinger systems set on
$H^1(\mathbb{R}^N)\times H^1(\mathbb{R}^N)$. In \cite{TZZ20}, the authors proved the existence of solutions $(\lambda_1,\lambda_2,u,v)\in\mathbb{R}^2\times H^1(\mathbb{R}^3)\times H^1(\mathbb{R}^3)$ to systems of coupled Schr\"{o}dinger equations.

The existence of normalized solution for fractional Schr\"{o}dinger system is an interesting problem. The fractional Schr\"{o}dinger  equation is introduced by Laskin \cite{L1,L2} through expanding the Feynman path integral from Brownian-like to L\'{e}vy-like mechanical paths. The path integral over the L\'{e}vy-like quantum-mechanical paths allows to develop the generalization of the quantum mechanics.

The fractional Laplacian  $(-\Delta)^s$ with $s\in (0,1)$ of a function $f:\mathbb{R}^N\rightarrow \mathbb{R}$ is expressed by the formula
\begin{equation*}
(-\Delta)^sf(x)=C_{N,s}P.V.\int_{\mathbb{R}^N}\frac{f(x)-f(z)}{|x-z|^{N+2s}}dz,
\end{equation*}
P.V. stands for the Cauchy principal value, and $C_{N,s}$ is a normalization constant.

It can also be defined as a pseudo-differential operator
\begin{equation*}
\mathscr{F}((-\Delta)^sf)(\xi)=|\xi|^{2s}\mathscr{F}(f)(\xi)=|\xi|^{2s}\hat{f}(\xi),
\end{equation*}
where $\mathscr{F}$ is the Fourier transform. For more details about the fractional Laplacian we refer to \cite{CS07,FL2013,FLS16,RS14,S07} and the references therein. The nature function space associated with $(-\Delta)^s$ in $N$ dimension is
\begin{equation*}
H^s(\mathbb{R}^N):=\bigg\{u~\big| \int_{\mathbb{R}^{2N}}\frac{|u(x)-u(z)|^2}{|x-z|^{N+2s}}dxdz<+\infty~\text{and } \int_{\mathbb{R}^N}|u(x)|^2dx<+\infty\bigg\},
\end{equation*}
equipped with norm
\begin{equation*}
\|u\|_{H^s(\mathbb{R}^N)}=\left(\int_{\mathbb{R}^N}|(-\Delta)^\frac{s}{2}u|^2dx+\int_{\mathbb{R}^N}u^2dx\right)^{\frac12},
\end{equation*}
where, by Fourier transform
\begin{align*}
\int_{\mathbb{R}^N}|(-\Delta)^\frac{s}{2}u|^2dx&=\int_{\mathbb{R}^N}|\xi|^{2s}\hat{u}(\xi)^2d\xi=\int_{\mathbb{R}^N}(-\Delta)^s u\cdot udx\\
&=C_{N,s}\int_{\mathbb{R}^N}\int_{\mathbb{R}^N}\frac{(u(x)-u(z))u(x)}{|x-z|^{N+2s}}dzdx\\
&=\frac{C_{N,s}}{2}\int_{\mathbb{R}^{2N}}\frac{|u(x)-u(z)|^2}{|x-z|^{N+2s}}dzdx.
\end{align*}

The energy functional associated with \eqref{GPE} is
\begin{align}\label{E-GPE}
E(u,v)
=\frac12\int_{\mathbb{R}^N}(|(-\Delta)^\frac{s}{2}u|^2+|(-\Delta)^\frac{s}{2}v|^2)dx
-\frac{1}{2p}\int_{\mathbb{R}^N}(\mu_1|u|^{2p}+2\beta |u|^p|v|^p+\mu_2|v|^{2p})dx
\end{align}
on the constraint $H_{a_1}\times{H_{a_2}}$, where for $0<a\in\mathbb{R}$, we define
\begin{align*}
H_a:=\{u\in H^s(\mathbb{R}^N):\int_{\mathbb{R}^N}u^2dx=a^2\}.
\end{align*}

We prove the existence of normalized solution for different ranges of the coupling parameter $\beta>0$. Our main theorems are as follows:

\begin{theorem}\label{t1.1}
Assume $0<s<1$, $2s<N\le 4s$ and $1+\frac{2s}{N}<p<\frac{N}{N-2s}$. Let $a_1,~a_2,~\mu_1~\text{and}~\mu_2>0$ be fixed, and let $\beta_1>0$ be defined by
\begin{align}\label{3.1}
\begin{split}
&\max\Bigg\{\frac{1}{a_1^\frac{4ps-2(p-1)N}{(p-1)N-2s}\mu_1^\frac{2s}{(p-1)N-2s}},\frac{1}{a_2^\frac{4ps-2(p-1)N}{(p-1)N-2s}\mu_2^\frac{2s}{(p-1)N-2s}}\Bigg\}\\
\quad&=\frac{1}{a_1^\frac{4ps-2(p-1)N}{(p-1)N-2s}(\mu_1+\beta_1)^\frac{2s}{(p-1)N-2s}}+\frac{1}{a_2^\frac{4ps-2(p-1)N}{(p-1)N-2s}(\mu_2+\beta_1)^\frac{2s}{(p-1)N-2s}}.
\end{split}
\end{align}
If $0<\beta<\beta_1$, then \eqref{GPE} has a solution $(\tilde{\lambda}_1,\tilde{\lambda}_2,\tilde{u},\tilde{v})$ with $(\tilde{u},\tilde{v})$ on the constraint $H_{a_1}\times{H_{a_2}}$, such that $\tilde{\lambda}_1,\tilde{\lambda}_2>0$ and $\tilde{u}$ and $\tilde{v}$ are both positive and radial.
\end{theorem}

For the next result, we introduce a Pohozaev-type constraint as follows:
\begin{align}\label{eq1.3}
{F}:=\{(u,v)\in{H_{a_1}}\times{H_{a_2}}:G(u,v)=0\},
\end{align}
where
\begin{align}
G(u,v)
=\int_{\mathbb{R}^N}(|(-\Delta)^\frac{s}{2}u|^2+|(-\Delta)^\frac{s}{2}v|^2)dx
-\frac{(p-1)N}{2ps}\int_{\mathbb{R}^N}(\mu_1|u|^{2p}+2\beta |u|^p|v|^p+\mu_2|v|^{2p})dx.
\end{align}

We define a Rayleigh-type quotient as
\begin{equation}\label{eq1.4}
\mathcal{R}(u,v):=\frac{R_0\bigg(\int_{\mathbb{R}^N}(|(-\Delta)^\frac{s}{2} u|^2+|(-\Delta)^\frac{s}{2}v|^2)dx\bigg)^\frac{(p-1)N}{(p-1)N-2s}}{\bigg(\int_{\mathbb{R}^N}(\mu_1|u|^{2p}+2\beta|u|^p|v|^p+\mu_2|v|^{2p})dx\bigg)^\frac{2s}{(p-1)N-2s}},
\end{equation}where

\begin{equation}\label{eq1.1}
R_0=\frac{(p-1)N-s}{2(p-1)N}\Bigg(\frac{2ps}{(p-1)N}\Bigg)^\frac{2s}{(p-1)N-2s}.
\end{equation}

\begin{theorem}\label{t1.2}
Assume $0<s<1$, $2s<N\le 4s$ and $1+\frac{2s}{N}<p<\frac{N}{N-2s}$. Let $a_1,a_2,\mu_1~\text{and}~\mu_2>0$ be fixed, and let $\beta_2>0$ be defined by
\begin{align}\label{eq1.2}
\begin{split}
&\min\Bigg\{\frac{1}{a_1^\frac{4ps-2(p-1)N}{(p-1)N-2s}\mu_1^\frac{2s}{(p-1)N-2s}},\frac{1}{a_2^\frac{4ps-2(p-1)N}{(p-1)N-2s}\mu_2^\frac{2s}{(p-1)N-2s}}\Bigg\}\\
\quad&=\frac{(a_1^2+a_2^2)^\frac{(p-1)N}{(p-1)N-2s}}{(\mu_1a_1^{2p}+2\beta _2a_1^pa_2^p+\mu_2a_2^{2p})^\frac{2s}{(p-1)N-2s}}.
\end{split}
\end{align}

If $\beta>\beta_2$, then \eqref{GPE} has a solution $(\bar \lambda_1,\bar \lambda_2,\bar u,\bar v)$ with $(\bar u,\bar v)$ on the constraint $H_{a_1}\times{H_{a_2}}$, such that $\bar \lambda_1,\bar \lambda_2>0$ and $\bar u$ and $\bar v$ are both positive and radial. Moreover,  $(\bar \lambda_1,\bar \lambda_2,\bar u,\bar v)$ is a solution in the sense that
\begin{align*}
E(\bar u,\bar v)&=\inf\{E(u,v):(u,v)\in F\}=\inf_{(u,v)\in{H_{a_1}}\times{H_{a_2}}}\mathcal{R}(u,v)
\end{align*}
holds.
\end{theorem}

In the system \eqref{GPE} with prescribed $L^2$ constraint, the problem appears to be more complicated as the Lagrange multipliers $\lambda_i$ are also need to be determined simultaneously. The exponent $2p\in (2+\frac{4s}{N},\frac{2N}{N-2s})$ brings another difficulty as it is $L^2$-supercritical and $E(u,v)$ is unbounded from below on the $L^2$ constraint. To overcome these difficulties, the idea introduced by L. Jeanjean in \cite{J97,BJN18} can be adopted to our system: A minimax argument can be applied to $E$, allowing one to construct a Palais-Smale sequence on the constraint satisfying the Pohozaev identity in limit sense. This leads to the boundedness of Palais-Smale sequence. Some a priori estimates on $\lambda_i$ and a Liouville-type result for fractional Laplacian (Lemma \ref{l2.3}) ensure $H^s$-convergence of the Palais-Smale sequence.

We don't know if the results are still true in high dimensions. Since $u\in H^s(\mathbb R^N)$, when the Liouville-type result is applied, we require that $2\le \frac{N}{N-2s}$ to get our results. It should be interesting to consider the problem in high dimension, even in the Laplacian case.

The paper is organized as follows. In Section \ref{sec2}, we introduce some important lemmas. In Section \ref{sec3}, we prove Theorem \ref{t1.1} and in Section \ref{sec4}, the proof of Theorem \ref{t1.2} is given.

\section{Preliminaries}\label{sec2}
In this section, we will show some facts about the fractional NLS equation, which are used later. First, we need the important fractional Gagliardo-Nirenberg-Sobolev inequality:
\begin{equation}\label{Gagliardo-Nirenberg-Sobolev inequality}
\int_{\mathbb{R}^N}|u|^{\alpha+2}dx
\leqslant C_{opt}\big(\int_{\mathbb{R}^N}|(-\Delta)^{\frac{s}{2}}u|^2dx\big)^\frac{N\alpha}{4s}\big(\int_{\mathbb{R}^N}|u|^2dx\big)^{\frac{\alpha(2s-N)}{4s}+1}.
\end{equation}
Here $\alpha>0$ and $C_{opt}>0$ denotes the optimal constant depending only on $\alpha$, $N$ and $s$.

It is well known that when $N>2s$,
\begin{equation}\label{embedding}
  H^s(\mathbb{R}^N)\hookrightarrow L^p(\mathbb{R}^N),~\text{for all}~ 2\leqslant p\leqslant \frac{2N}{N-2s}.
\end{equation}

Consider the general fractional Laplacian equation
\begin{equation}\label{2.1}
(-\Delta)^s u=f(u)\quad \text{in}~\mathbb{R}^N
\end{equation}
with $f\in C^2(\mathbb{R})$. Assume that $u\in{H}^s(\mathbb{R}^N)\cap L^\infty(\mathbb{R}^N)$ is a solution to \eqref{2.1}, the Pohozaev identity for \eqref{2.1} is proved in \cite{BM17}.
\begin{theorem}\label{t2.1}\cite{BM17}
Let $u\in{H}^s(\mathbb{R}^N)\cap L^\infty(\mathbb{R}^N)$ be a solution to \eqref{2.1} and $F(u)\in L^1(\mathbb{R}^N)$. Then,
\begin{equation}\label{Pohozaev indentity}
(N-2s)\int_{\mathbb{R}^N} uf(u)dx=2N\int_{\mathbb{R}^N}F(u)dx,
\end{equation}
where $F(u)=\int^u_0f(t)dt$.
\end{theorem}

Let us consider the scalar problem
\begin{equation}\label{2.2}
\left\{
\begin{aligned}
&(-\Delta)^s w+w=|w|^{2p-2}w\qquad\quad\text{in}~{\mathbb{R}^N},\\
&w>0\qquad\qquad\qquad\qquad\qquad\quad\text{in}~{\mathbb{R}^N},\\
&w(0)=\max_{x\in\mathbb{R}^N} w\quad \text{and} \quad w\in H^s(\mathbb{R}^N).\\
\end{aligned}
\right.
\end{equation}
It is shown in \cite{FLS16} that there is a unique positive radial solution $w_0\in H^s(\mathbb{R}^N)\cap L^{\infty}(\mathbb{R}^N) $ to \eqref{2.2} for $1<p<\frac{N}{N-2s}$ and $N>2s$, see Proposition 3.1 in \cite{FLS16}.

We set
\begin{align}\label{2.4}
C_0:=\int_{\mathbb{R}^N} w_0^2dx\quad \text{and} \quad C_1:=\int_{\mathbb{R}^N} w_0^{2p}dx.
\end{align}

By the Pohozaev identity for \eqref{2.2}, we can get
\begin{equation}\label{2.3}
\int_{\mathbb{R}^N}|(-\Delta)^\frac{s}{2}w_0|^2dx=\frac{(p-1)N}{2ps}\int_{\mathbb{R}^N}w_0^{2p}dx=\frac{(p-1)NC_1}{2ps}.
\end{equation}

\begin{remark}
For the constant $C_{opt}$ in Gagliardo-Nirenberg-Sobolev inequality \eqref{Gagliardo-Nirenberg-Sobolev inequality} with $\alpha=2p-2$, it can be evaluated by $w_0$
\begin{align*}
\frac{1}{C_{opt}}
&=\inf_{u\in H^s\setminus\{0\}}
\frac{\big(\int_{\mathbb{R}^N}|(-\Delta)^{\frac{s}{2}}u|^2dx\big)^\frac{(p-1)N}{2s}\big(\int_{\mathbb{R}^N}|u|^2dx\big)^{\frac{2ps-(p-1)N}{2s}}}{\int_{\mathbb{R}^N}|u|^{2p}dx}\nonumber\\
&=\frac{\big(\int_{\mathbb{R}^N}|(-\Delta)^{\frac{s}{2}}w_0|^2dx\big)^\frac{(p-1)N}{2s}C_0^{\frac{2ps-(p-1)N}{2s}}}{C_1}\nonumber\\
&=\frac{\bigg(\frac{(p-1)N}{2ps}C_1\bigg)^\frac{(p-1)N}{2s}C_0^{\frac{2ps-(p-1)N}{2s}}}{C_1}\nonumber\\
&=\frac{C_0^\frac{2ps-(p-1)N}{2s}C_1^{\frac{(p-1)N-2s}{2s}}}{(\frac{2ps}{(p-1)N})^\frac{(p-1)N}{2s}},
\end{align*}
which implies that
\begin{align}\label{the optimal constant}
C_{opt}=\frac{(\frac{2ps}{(p-1)N})^\frac{(p-1)N}{2s}}{C_0^\frac{2ps-(p-1)N}{2s}C_1^\frac{(p-1)N-2s}{2s}}.
\end{align}
\end{remark}

For $a,\mu>0$ fixed, we search for $(\lambda,w)\in \mathbb{R}\times H^s(\mathbb{R}^N)$, with $\lambda>0$ in $\mathbb{R}$, solving
\begin{equation}\label{2.5}
\left\{
\begin{aligned}
&(-\Delta)^s w+\lambda w=\mu |w|^{2p-2}w,\quad\text{in}~{\mathbb{R}^N}\\
&w>0\qquad\qquad\qquad\qquad\qquad\quad\text{in}~{\mathbb{R}^N},\\
&w(0)=\max_{x\in\mathbb{R}^N} w\quad \text{and} \quad \int_{\mathbb{R}^N}w^2dx=a^2.\\
\end{aligned}
\right.
\end{equation}
Solution to \eqref{2.5} can be found as the critical points of $I_\mu:H^s(\mathbb{R}^N)\rightarrow \mathbb{R}$, defined by
\begin{equation}\label{2.6}
I_\mu(w)=\int_{\mathbb{R}^N}\left(\frac12|(-\Delta)^\frac{s}{2}w|^2-\frac{\mu}{2p}|w|^{2p}\right)dx,
\end{equation}
constrained on the $L^2$-sphere $H_a:=\{u\in H^s(\mathbb{R}^N):\int_{\mathbb{R}^N}u^2=a^2\}$, and $\lambda$ appears as the Lagrange multiplier. It is well known that it can be obtained from $w_0$ by scaling.

\begin{lemma}\label{l2.5}
Equation \eqref{2.5} has a unique positive solution $(\lambda_{a,\mu},w_{a,\mu})$ defined by
\begin{equation*}
\lambda_{a,\mu}:=\Bigg[\frac{1}{\mu}(\frac{C_0}{a^2})^{p-1}\Bigg]^\frac{2s}{(p-1)N-2s},\quad
w_{a,\mu}:=\Bigg(\frac{C_0^{2s}}{\mu^N a^{4s}}\Bigg)^\frac{1}{2(p-1)N-4s}w_0\Bigg(\lambda_{a,\mu}^{\frac{1}{2s}}x\Bigg).
\end{equation*}
Furthermore, $w_{a,\mu}$ satisfies
\begin{equation}\label{2.27}
\int_{\mathbb{R}^N}|(-\Delta)^\frac{s}{2} w_{a,\mu}|^2dx=\frac{(p-1)N}{2ps}\frac{C_1C_0^{\frac{2ps-(p-1)N}{(p-1)N-2s}}}{\mu^\frac{2s}{(p-1)N-2s}a^\frac{4ps-2(p-1)N}{(p-1)N-2s}},
\end{equation}
\begin{equation}\label{2.28}
\int_{\mathbb{R}^N}|w_{a,\mu}|^{2p}dx=\frac{C_1C_0^{\frac{2ps-(p-1)N}{(p-1)N-2s}}}{\mu^\frac{(p-1)N}{(p-1)N-2s}a^\frac{4ps-2(p-1)N}{(p-1)N-2s}},
\end{equation}
\begin{equation}\label{2.29}
I_{\mu}(w_{a,\mu})=\frac{(p-1)N-2s}{4ps}\frac{C_1C_0^{\frac{2ps-(p-1)N}{(p-1)N-2s}}}{\mu^\frac{2s}{(p-1)N-2s}a^\frac{4ps-2(p-1)N}{(p-1)N-2s}}.
\end{equation}
\end{lemma}
\begin{proof}
 We can directly check that $w_{a,\mu}$ satisfies the equation \eqref{2.5} with $\lambda=\lambda_{a,\mu}$ and $w_{a,\mu}$ is the unique positive radial solution of \eqref{2.5} by \cite{FLS16}.
By direct calculation,
\begin{align*}
\int_{\mathbb{R}^N}|w_{a,\mu}|^{2p}dx=\frac{C_0^{\frac{2ps-(p-1)N}{(p-1)N-2s}}}{\mu^\frac{(p-1)N}{(p-1)N-2s}a^\frac{4ps-2(p-1)N}{(p-1)N-2s}}\int_{\mathbb{R}^N}|w_0|^{2p}dx,
\end{align*}
we get the equality \eqref{2.28}.

\begin{align*}
\int_{\mathbb{R}^N}|(-\Delta)^\frac{s}{2} w_{a,\mu}|^2dx=\frac{C_0^{\frac{2ps-(p-1)N}{(p-1)N-2s}}}{\mu^\frac{2s}{(p-1)N-2s}a^\frac{4ps-2(p-1)N}{(p-1)N-2s}}\int_{\mathbb{R}^N}|(-\Delta)^\frac{s}{2} w_0|^2dx,
\end{align*}
combined with \eqref{2.3}, we get the equality \eqref{2.27}. Combining \eqref{2.27} and \eqref{2.28} together, we obtain the equality \eqref{2.29}.
\end{proof}

Let us introduce the set
\begin{equation}\label{2.7}
\mathcal{P}(a,\mu):=\left\{w\in H_a:\int_{\mathbb{R}^N}|(-\Delta)^\frac{s}{2}w|^2dx=\frac{(p-1)N\mu}{2ps}\int_{\mathbb{R}^N}|w|^{2p}dx\right\}.
\end{equation}

When $1+\frac{2s}{N}<p<\frac{N}{N-2s}$, we have the following lemma.
\begin{lemma}\label{l2.4}
Assume that $1+\frac{2s}{N}<p<\frac{N}{N-2s}$, if $w$ is a solution of \eqref{2.5}, then $w\in \mathcal{P}(a,\mu)$. In addition the positive solution of \eqref{2.5} minimizes $I_\mu$ on $\mathcal{P}(a,\mu)$.
\end{lemma}
\begin{proof}
Let $(w,\lambda)\in H_a\times\mathbb{R}$ be a solution of \eqref{2.5}. By Pohozaev identity \eqref{Pohozaev indentity},
\begin{align*}
(N-2s)\int_{\mathbb{R}^N}|(-\Delta)^\frac{s}{2}w|^2dx
&=2N\left(-\frac{\lambda}{2}\int_{\mathbb{R}^N}|w|^2dx+\frac{\mu}{2p}\int_{\mathbb{R}^N}|w|^{2p}dx\right),
\end{align*}
combined with
\begin{equation*}
\int_{\mathbb{R}^N}|(-\Delta)^\frac{s}{2}w|^2dx+\lambda\int_{\mathbb{R}^N}w^2dx=\mu\int_{\mathbb{R}^N}|w|^{2p}dx,
\end{equation*}
we get
\begin{equation}\label{2.10}
\int_{\mathbb{R}^N}|(-\Delta)^\frac{s}{2}w|^2dx=\frac{(p-1)N\mu}{2ps}\int_{\mathbb{R}^N}|w|^{2p}dx,
\end{equation}
thus, $w\in \mathcal{P}(a,\mu)$.

In the following, we prove that the positive solution $w_{a,\mu}$ of \eqref{2.5} minimizes $I_\mu$ on $\mathcal{P}(a,\mu)$. For any $u\in \mathcal{P}(a,\mu)$, by Gagliardo-Nirenberg-Sobolev inequality \eqref{Gagliardo-Nirenberg-Sobolev inequality} and the fact that $\|u\|_{L^2}=a$, we have
\begin{align}\label{2.15}
\int_{\mathbb{R}^N}|u|^{2p}dx\leqslant C_{opt}a^{\frac{2ps-(p-1)N}{s}}\big(\int_{\mathbb{R}^N}|(-\Delta)^{\frac{s}{2}}u|^2dx\big)^\frac{(p-1)N}{2s}.
\end{align}

Together with \eqref{2.7}, we obtain
\begin{equation*}
\big(\int_{\mathbb{R}^N}|(-\Delta)^{\frac{s}{2}}u|^2dx\big)^\frac{(p-1)N-2s}{2s}\geq \frac{2ps}{(p-1)N\mu C_{opt}a^{\frac{2ps-(p-1)N}{s}}}.
\end{equation*}

Therefore, for any $u\in \mathcal{P}(a,\mu)$,
\begin{equation}\label{2.151}
\begin{array}{rl}
I_{\mu}(u)&=\frac{(p-1)N-2s}{2(p-1)N}\int_{\mathbb{R}^N}|(-\Delta)^{\frac{s}{2}}u|^2dx\\
\\
&\geq \frac{(p-1)N-2s}{2(p-1)N}\left(\frac{2ps}{(p-1)N\mu C_{opt}a^{\frac{2ps-(p-1)N}{s}}}\right)^{\frac{2s}{(p-1)-2s}}.
\end{array}
\end{equation}

It is clear that equality in \eqref{2.151} is obtained by $w_{a,\mu}$ due to Pohozaev identity \eqref{Pohozaev indentity} and the fact that $C_{opt}$ is achieved by $w_{a,\mu}$ (see \cite{FLS16}). Therefore
\begin{equation*}
I_{\mu}(w_{a,\mu})=\inf\limits_{u\in \mathcal{P}(a,\mu)} I_\mu(u).
\end{equation*}
\end{proof}

\begin{lemma}\label{l2.7}
For $1+\frac{2s}{N}<p<\frac{N}{N-2s}$, let $u\in H_a$ be arbitrary but fixed. Define $(l\star u)(x):=e^\frac{Nsl}{2}u(e^{sl}x)$, then we have
\begin{itemize}
\item[(i)] $\|(-\Delta)^\frac{s}{2} (l\star u)\|_{L^2}\rightarrow 0$ and $I_{\mu}(l\star u)\rightarrow 0$ as $l\rightarrow -\infty$,
\item[(ii)] $\|(-\Delta)^\frac{s}{2} (l\star u)\|_{L^2}\rightarrow +\infty$ and $I_{\mu}(l\star u)\rightarrow -\infty$ as $l\rightarrow +\infty$,
\item[(iii)] $f_u(l)=I_\mu(l\star u)$ reaches its unique maximum value at $l(u)\in \mathbb{R}$ with $l(u)\star u\in \mathcal{P}(a,\mu)$.
\end{itemize}
\end{lemma}

\begin{proof}
By direct calculation, we have
\begin{equation*}
\|l\star u\|_{L^2}=a \quad\text{and}\quad \|(-\Delta)^\frac{s}{2} (l\star u)\|_{L^2}=e^{s^2l}\|(-\Delta)^\frac{s}{2} u\|_{L^2},
\end{equation*}
thus, $\|(-\Delta)^\frac{s}{2} (l\star u)\|_{L^2}\rightarrow 0$ as $l\rightarrow -\infty$, and $\|(-\Delta)^\frac{s}{2} (l\star u)\|_{L^2}\rightarrow +\infty$ as $l\rightarrow +\infty$.

Now we compute $f_u(l)$,
\begin{equation}\label{2.13}
\begin{array}{rl}
f_u(l)&=I_\mu(l\star u)=\int_{\mathbb{R}^N}(\frac12|(-\Delta)^\frac{s}{2}(l\star u)|^2-\frac{\mu}{2p}|l\star u|^{2p})dx\\
\\
&=\frac{e^{2s^2l}}{2}\|(-\Delta)^\frac{s}{2}u\|^2_{L^2}-\frac{e^{(p-1)Nsl}}{2p}\mu\|u\|^{2p}_{L^{2p}},
\end{array}
\end{equation}
thus, $I_{\mu}(l\star u)\rightarrow 0$ as $l\rightarrow -\infty$. Due to $p>1+\frac{2s}{N}$,  we have $I_{\mu}(l\star u)\rightarrow -\infty$ as $l\rightarrow +\infty$. (i),(ii) are proved. To show the third claim, by \eqref{2.13}, we have
\begin{equation}\label{2.24}
\begin{array}{rl}
f'_u(l)&=s^2e^{2s^2l}\|(-\Delta)^\frac{s}{2}u\|^2_{L^2}-\frac{(p-1)Ns\mu}{2p}{e^{(p-1)Nsl}}\|u\|_{L^{2p}}^{2p}\\
\\
&=s^2\|(-\Delta)^\frac{s}{2}(l\star u)\|^2_{L^2}-\frac{(p-1)Ns\mu}{2p}\|l\star u\|_{L^{2p}}^{2p}.
\end{array}
\end{equation}

Therefore $f'_u(l)=0$ is equivalent to
\begin{align}\label{2.20}
e^{s[(p-1)N-2s]l}=\frac{\|(-\Delta)^\frac{s}{2}u\|^2_{L^2}}{\frac{(p-1)N\mu}{2ps}\|u\|_{L^{2p}}^{2p}}.
\end{align}

So there exists a unique $l_0\in\mathbb{R}$ such that $f'_u(l)|_{l=l_0}=0$ and $l_0\star u\in \mathcal{P}(a,\mu)$. Furthermore, we have
\begin{align*}
f''_u(l)|_{l=l_0}&=\left(2s^4e^{2s^2l}\|(-\Delta)^\frac{s}{2}u\|^2_{L^2}-\frac{{(p-1)}^2N^2s^2\mu}{2p}{e^{(p-1)Nsl}}\|u\|_{L^{2p}}^{2p}\right)\Big|_{l=l_0}\nonumber\\
&=\Bigg(2s-(p-1)N\Bigg)\frac{e^{(p-1)Nsl_0}}{2p}s^2(p-1)N\mu\|u\|_{L^{2p}}^{2p}\nonumber\\
&<0.
\end{align*}
Note that,
\begin{align*}
f'_u(l)=
\left\{\begin{array}{cc}>0 & \text{if}~ l<l_0, \\=0 & \text{if}~  l=l_0, \\< 0 & \text{if}~ l>l_0.
\end{array}\right.
\end{align*}
This implies that $f_u(l)$ gets its unique maximum value at $l_0(u)$. If $u\in\mathcal{P}(a,\mu)$, then by \eqref{2.20}, $l_0=0$.
\end{proof}

When $\mu_0=({C_0}/{a^2})^{p-1}$ in \eqref{2.5}, by Lemma \ref{l2.5}, $\lambda_{a,\mu_0}=1$, i.e., $w_{a,\mu_0}$ is the unique radial solution of the following equation:
\begin{equation}\label{eq2.1}
\left\{\begin{aligned}
&(-\Delta)^s{w}+w=\mu_0w^{2p-1}~~~~\text{in}~\mathbb{R}^N,\\
&w(0)=\max_{x\in\mathbb{R}^N}{w},~\text{and}~\int_{\mathbb{R}^N}{w^2}dx=a^2,
\end{aligned}
\right.
\end{equation}
and hence is a minimizer of $I_{\mu_0}$ on $\mathcal{P}(a,\mu_0)$. Our next result shows that this level can also be characterized as an infimum of a Rayleigh-type quotient.

\begin{lemma}\label{l2.6}
\begin{equation}\label{2.30}
\inf_{u\in\mathcal{P}(a,\mu_0)}I_{\mu_0}(u)=\inf_{u\in H_a}\mathcal{R}(u),
\end{equation}
where
\begin{equation*}
\mathcal{R}(u):=\frac{R_0(\int_{\mathbb{R}^N}|(-\Delta)^\frac{s}{2} u|^2dx)^\frac{(p-1)N}{(p-1)N-2s}}{(\mu_0\int_{\mathbb{R}^N}|u|^{2p}dx)^\frac{2s}{(p-1)N-2s}},
\end{equation*}and $R_0$ is defined in \eqref{eq1.1}.
\end{lemma}
\begin{proof}
If $u\in \mathcal{P}(a,\mu_0)$, then
\begin{equation}\label{2.31}
\frac{2ps\int_{\mathbb{R}^N}|(-\Delta)^\frac{s}{2}u|^2dx}{(p-1)N\mu_0\int_{\mathbb{R}^N}|u|^{2p}dx}=1\quad\text{and}
\quad I_{\mu_0}(u)=(\frac12-\frac{s}{(p-1)N})\int_{\mathbb{R}^N}|(-\Delta)^\frac{s}{2}u|^2dx.
\end{equation}
Therefore,
\begin{equation*}
\begin{array}{rl}
I_{\mu_0}(u)&=(\frac12-\frac{s}{(p-1)N})\int_{\mathbb{R}^N}|(-\Delta)^\frac{s}{2}u|^2dx\left(\frac{2ps\int_{\mathbb{R}^N}|(-\Delta)^\frac{s}{2}u|^2dx}{(p-1)N\mu_0\int_{\mathbb{R}^N}|u|^{2p}dx}\right)^\frac{2s}{(p-1)N-2s},\\
&=\mathcal{R}(u),
\end{array}
\end{equation*}
which proves that
\begin{equation*} \inf_{u\in\mathcal{P}(a,\mu_0)}I_{\mu_0}(u)\geqslant\inf_{u\in H_a}\mathcal{R}(u).
\end{equation*}

On the other hand, for all $l\in\mathbb{R}$ and $u\in H_a$, direct calculation shows that
\begin{equation*}
\mathcal{R}(u)=\mathcal{R}(l\star u).
\end{equation*}
By Lemma \ref{l2.7}, we know that for $u\in H_a$ be arbitrary but fixed, there exists a unique $l_0(u)\in \mathbb{R}$ such that  $l_0(u)\star u\in \mathcal{P}(a,\mu_0)$, and $I_{\mu_0}(l\star u)$ reaches its unique maximum at $l_0(u)\star u$. Hence, for every $u\in H_a$, we have
\begin{align*}
\mathcal{R}(u)=\mathcal{R}(l_0(u)\star u)=I_{\mu_0}(l_0(u)\star u)\geqslant \inf_{v\in\mathcal{P}(a,\mu_0)}I_{\mu_0}(v),
\end{align*}
which proves that
\begin{align*}
 \inf_{u\in\mathcal{P}(a,\mu_0)}I_{\mu_0}(u)\leqslant\inf_{u\in H_a}\mathcal{R}(u).
\end{align*}
\end{proof}

Next, we give a Liouville-type result for fractional Laplacian. Similar Liouville-type result for Laplacian can be found in \cite{I14}.
\begin{lemma}\label{l2.3}
Let $u\in{H^s(\mathbb{R}^N)}$ with $N>2s$,
\begin{itemize}
\item[(i)]If $u$ satisfies
\begin{equation*}
\left\{
\begin{aligned}&(-\Delta)^s{u}\geq0\qquad\quad\quad\text{in}~{\mathbb{R}^N},\\
&u\in{L^q({\mathbb{R}^N})},~q\in{(0,\frac{N}{N-2s}]},\\
&u\geq0,
\end{aligned}
\right.
\end{equation*}
then $u\equiv0$.
\item[(ii)]If $u$ satisfies
\begin{equation*}
\left\{\begin{aligned}&(-\Delta)^s{u}\geq{u^q}\qquad\quad\quad\text{in}~{\mathbb{R}^N},\\
&u\geq 0,~\text{and}~q\in(1,\frac{N}{N-2s}],
\end{aligned}
\right.
\end{equation*}
then $u\equiv0$.
\end{itemize}
\end{lemma}
\begin{proof}
We prove (i) by contradiction. If $u\not\equiv0$, by maximum principle, we have $u>0~\text{in}~{\mathbb{R}^N}$. Let $v(x)=\frac{1}{|x|^{N-2s}}u(\frac{x}{|x|^2})$. Then $v(x)>0$ in $\mathbb{R}^N\setminus\{0\}$, and $v(x)$ satisfies
\begin{equation*}
(-\Delta)^s{v(x)}=\frac{1}{|x|^{N+2s}}(-\Delta)^s{u(\frac{x}{|x|^2})}~~~~\text{in}~{\mathbb{R}^N}\setminus\{0\},
\end{equation*}
so $(-\Delta)^s{v}\geq0$ in distribution sense. Since $u\in H^s(\mathbb{R}^N)\subset{L_{2s}}(\mathbb{R}^N)$, where
$$
L_{2s}(\mathbb{R}^N)=\{w(x):{\mathbb{R}^N}\rightarrow{\mathbb{R}}~|~\int_{\mathbb{R}^N}\frac{|w(x)|}{1+{|x|^{N+2s}}}dx<+\infty\},
$$ we can see that $v\in L_{2s}(\mathbb{R}^N)$.  By Theorem 1 in \cite{LC2018}, there exists a constant $C>0$ such that
\begin{equation*}
\inf_{|x|<\frac12}v(x)\geq{C}.\\
\end{equation*}
Therefore, we obtain that
\begin{equation*}
u(x)\geq\frac{C}{|x|^{N-2s}},~|x|>2.
\end{equation*}
For $q\in(0,\frac{N}{N-2s}]$, we can compute
\begin{equation*}
\int_{\mathbb{R}^N}u^qdx\geq C\int_{|x|>2}\frac{1}{|x|^{(N-2s)q}}dx\geq C\int_{|x|>2}\frac{1}{|x|^{N}}dx=+\infty,
\end{equation*}
which is a contradiction to $u\in L^q({\mathbb{R}^N})$. So $u\equiv0$.

To prove (ii), let $\varphi$ be the first eigenfunction of
\begin{equation*}
\left\{\begin{aligned}
&(-\Delta)^s{\varphi}={\lambda_1}{\varphi}~~~\text{in}~{B_1(0)},\\
&{\varphi}\equiv0~~~~~~\text{in}~{B^c_{1}(0)},
\end{aligned}\right.
\end{equation*}
where $B_1(0)$ is the unit ball in ${\mathbb{R}^N}$, $\varphi>0$ in $B_1(0)$ and $\lambda_1>0$ is the first eigenvalue of $(-\Delta)^s$ in $B_1(0)$. For any $R>0$ but fixed, let $\varphi_{R}(x)=\varphi(\frac{x}{R})$. Then
\begin{equation*}
\left\{\begin{aligned}&(-\Delta)^s{\varphi_{R}}=R^{-2s}\lambda_1{\varphi_{R}}~~~\text{in}~B_R(0), \\
&{\varphi_{R}}\equiv0~~~~~~\text{in}~B^c_{R}(0).
\end{aligned}\right.
\end{equation*}
We can compute
\begin{align*}
\int_{B_R(0)}u^q{\varphi_{R}}dx&=\int_{\mathbb{R}^N}u^q{\varphi_{R}}dx\leq\int_{\mathbb{R}^N}(-\Delta)^su{\varphi_{R}}dx\\
&=\int_{\mathbb{R}^N}(-\Delta)^s{\varphi_{R}}udx=\int_{B_R(0)}uR^{-2s}\lambda_1{\varphi_{R}}dx+\int_{B^c_{R}(0)}(-\Delta)^s{\varphi_{R}}udx\\
&\leq\int_{B_R(0)}uR^{-2s}\lambda_1{\varphi_{R}}dx\leq{R^{-2s}\lambda_1}\Bigg(\int_{B_R(0)}u^q{\varphi_{R}}dx\Bigg)^{\frac{1}{q}}\Bigg(\int_{B_R(0)}{\varphi_{R}}dx\Bigg)^{1-\frac{1}{q}},
\end{align*}
in the above, we have use the fact that $(-\Delta)^s\varphi_R<0$ in $B_R^c(0)$. Therefore
\begin{equation}\label{eqc}
\int_{B_R(0)}u^q{\varphi_{R}}dx\leq{C}R^{-\frac{2sq}{q-1}}\int_{\mathbb{R}^N}\varphi_{R}dx\le CR^{\frac{N(q-1)-2sq}{q-1}}.
\end{equation}

When $q\in(1,\frac{N}{N-2s})$, we have
\begin{equation*}
\min_{B_{\frac12}(0)}\varphi\cdot\int_{B_{\frac{R}{2}}(0)}u^qdx\leq\int_{\mathbb{R}^N}u^q{\varphi_{R}}dx\leq{C}R^{\frac{N(q-1)-2sq}{q-1}}\rightarrow0,~\text{as}~R\rightarrow\infty.
\end{equation*}
So we have $u\equiv0$.

When $q=\frac{N}{N-2s}$, we have
\begin{equation*}
\min_{B_{\frac12}(0)}\varphi\cdot\int_{B_{\frac{R}{2}}(0)}u^qdx\leq\int_{\mathbb{R}^N}u^q{\varphi_{R}}dx\leq{C} \quad\text{for all $R>0$},
\end{equation*}
with $C$ independent of $R$ by \eqref{eqc}, so $u\in L^q(\mathbb{R}^N)$. By (i), we obtain $u\equiv0$.
\end{proof}

\section{Proof of Theorem \ref{t1.1}}\label{sec3}

In this section, we give the proof of Theorem \ref{t1.1}. We work in a radial setting. That is, we find the critical point of the functional $E$ constrained on $H^{rad}_{a_1}\times{H^{rad}_{a_2}}$, where for any $a>0$, we define
\begin{equation*}
H^{rad}_{a}:=H_{a}\cap H_r^s(\mathbb{R}^N),
\end{equation*}
and $H^s_{r}(\mathbb{R}^N)$ is the subset of $H^s(\mathbb{R}^N)$ containing all the functions which are radial with respect to the origin. We know that $H^s_r(\mathbb{R}^N)\hookrightarrow L^p(\mathbb{R}^N)$ is compact when $2<p<\frac{2N}{N-2s}$. Due to the Palais principle of symmetric criticality, the critical points of $E$ constrained on $H^{rad}_{a_1}\times H^{rad}_{a_2}$ are true critical points of $E$ constrained in the full product $H_{a_1}\times H_{a_2}$.

\vspace{1.5mm}
For $a_1,a_2,\mu_1,\mu_2>0$, let $\beta_1>0$ be defined by \eqref{3.1}.
\begin{lemma}\label{l3.2}
For $0<\beta<\beta_1$, there holds:
\begin{equation*}
\inf\{E(u_1,u_2):(u_1,u_2)\in\mathcal{P}(a_1,\mu_1+\beta)\times\mathcal{P}(a_2,\mu_2+\beta)\}>\max\{I_{\mu_1}(w_{a_1,\mu_1}),I_{\mu_2}(w_{a_2,\mu_2})\},
\end{equation*}
where $I_{\mu_i}(w_{a_i,\mu_1})$, $i=1,2$ is defined by \eqref{2.29}.
\end{lemma}
\begin{proof}
For $(u_1,u_2)\in\mathcal{P}(a_1,\mu_1+\beta)\times\mathcal{P}(a_2,\mu_2+\beta)\}$, we have
\begin{align*}
E(u_1,u_2)=&\int_{\mathbb{R}^N}\left(\frac12|(-\Delta)^\frac{s}{2}u_1|^2-\frac{\mu_1}{2p}|u_1|^{2p}\right)dx\\
&+\int_{\mathbb{R}^N}\left(\frac12|(-\Delta)^\frac{s}{2}u_2|^2-\frac{\mu_2}{2p}|u_2|^{2p}\right)dx-\frac{\beta}{p}\int_{\mathbb{R}^N}u^p_1u^p_2dx\\
\geqslant& I_{\mu_1}(u_1)+I_{\mu_2}(u_2)-\frac{\beta}{2p}\int_{\mathbb{R}^N}u^{2p}_1dx-\frac{\beta}{2p}\int_{\mathbb{R}^N}u^{2p}_2dx\\
=&I_{\mu_1+\beta}(u_1)+I_{\mu_2+\beta}(u_2)\\
\geqslant& \inf_{u\in\mathcal{P}(a_1,\mu_1+\beta)}I_{\mu_1+\beta}(u)+\inf_{v\in\mathcal{P}(a_1,\mu_1+\beta)}I_{\mu_2+\beta}(v)\\
= &I_{\mu_1+\beta}(w_{a_1,\mu_1+\beta})+I_{\mu_2+\beta}(w_{a_2,\mu_2+\beta}),
\end{align*}
by Lemma \ref{l2.4}. From \eqref{2.29} and \eqref{3.1}, it is easy to get, when $0<\beta<\beta_1$, we have
\begin{align*}
&\max\{I_{\mu_1}(w_{a_1,\mu_1}),I_{\mu_2}(w_{a_2,\mu_2})\}\\
&=\max\{\frac{(p-1)N-2s}{4ps}\frac{C_1C_0^{\frac{2ps-(p-1)N}{(p-1)N-2s}}}{{\mu_1}^\frac{2s}{(p-1)N-2s}a_1^\frac{4ps-2(p-1)N}{(p-1)N-2s}}
,\frac{(p-1)N-2s}{4ps}\frac{C_1C_0^{\frac{2ps-(p-1)N}{(p-1)N-2s}}}{{\mu_2}^\frac{2s}{(p-1)N-2s}a_2^\frac{4ps-2(p-1)N}{(p-1)N-2s}}\}\\
&< I_{\mu_1+\beta}(w_{a_1,\mu_1+\beta})+I_{\mu_2+\beta}(w_{a_2,\mu_2+\beta}).
\end{align*}
Therefore,
\begin{align*}
&\inf\{E(u_1,u_2):(u_1,u_2)\in\mathcal{P}(a_1,\mu_1+\beta)\times\mathcal{P}(a_2,\mu_2+\beta)\}\\
&>\max\{I_{\mu_1}(w_{a_1,\mu_1}),I_{\mu_2}(w_{a_2,\mu_2})\}.
\end{align*}
\end{proof}

Now we fix $0<\beta<\beta_1$ and choose $\varepsilon>0$ such that
\begin{align}\label{3.2}
&\inf\{E(u_1,u_2):(u_1,u_2)\in\mathcal{P}(a_1,\mu_1+\beta)\times\mathcal{P}(a_2,\mu_2+\beta)\}\nonumber\\
&>\max\{I_{\mu_1}(w_{a_1,\mu_1}),I_{\mu_2}(w_{a_2,\mu_2})\}+\varepsilon .
\end{align}
Denote
\begin{align}\label{3.3}
w_1:=w_{a_1,\mu_1+\beta}\quad\text{and}\quad w_2:=w_{a_2,\mu_2+\beta},
\end{align}
and for $i=1,2$,
\begin{align}\label{3.4}
\varphi_i(l):=I_{\mu_i}(l\star w_i)\quad\text{and}\quad \tilde{\varphi}_i(l):=\frac{\partial}{\partial l}I_{\mu_i+\beta}(l\star w_i).
\end{align}

\begin{lemma}\label{l3.3}
For $i=1,2$, there exist $\rho_i<0$ and $R_i>0$, depending on $\varepsilon$ and $\beta$, such that
\begin{itemize}
\item[(i)] $0<\varphi_i(\rho_i)<\varepsilon$ and $\varphi_i(R_i)\leqslant 0$;
\item[(ii)] $\tilde{\varphi}_i(l)>0$ for any $l<0$,  $\tilde{\varphi}_i(0)=0$ and $\tilde{\varphi}_i(l)<0$ for any $l>0$. In particular, $\tilde{\varphi}_i(\rho_i)>0$ and $\tilde{\varphi}_i(R_i)<0$.
\end{itemize}
\end{lemma}
\begin{proof}
By Lemma \ref{l2.4} and Lemma \ref{l2.7}, we have
\begin{align*}
\varphi_i(l)&= \int_{\mathbb{R}^N}\left(\frac12|(-\Delta)^\frac{s}{2}(l\star w_i)|^2-\frac{\mu_i}{2p}|l\star w_i|^{2p}\right)dx      \nonumber\\
&=\frac{e^{2s^2l}}{2}\|(-\Delta)^\frac{s}{2}w_i\|^2_{L^2}-\frac{e^{(p-1)Nsl}}{2p}\mu_i\|w_i\|^{2p}_{L^{2p}}\nonumber\\
&=\Bigg(\frac{(p-1)N(\mu_i+\beta)}{2ps}\frac{e^{2s^2l}}{2}-\frac{e^{(p-1)Nsl}}{2p}\mu_i\Bigg)\|w_i\|^{2p}_{L^{2p}},
\end{align*}
thus, $\varphi_i(l)\rightarrow 0^+$ as $l\rightarrow -\infty$, and $\varphi_i(l)\rightarrow -\infty$ as $l\rightarrow +\infty$.  Therefore, there exist $\rho_i<0$ and $R_i>0$, such that $0<\varphi_i(\rho_i)<\varepsilon$ and $\varphi_i(R_i)\leqslant 0$.
\begin{align*}
\tilde{\varphi}_i(l)&=s^2e^{2s^2l}\|(-\Delta)^\frac{s}{2}w_i\|^2_{L^2}-\frac{e^{(p-1)Nsl}(p-1)N}{2p}s(\mu_i+\beta)\int_{\mathbb{R}^N}|w_i|^{2p}dx\nonumber\\
&=\Bigg(\frac{(p-1)N(\mu_i+\beta)}{2ps}s^2e^{2s^2l}-\frac{e^{(p-1)Nsl}(p-1)N}{2p}s(\mu_i+\beta)\Bigg)\int_{\mathbb{R}^N}|w_i|^{2p}dx\nonumber\\
&=\frac{(p-1)N(\mu_i+\beta)}{2p}se^{(p-1)Nsl}\Bigg(e^{(2s-(p-1)N)sl}-1\Bigg)\int_{\mathbb{R}^N}|w_i|^{2p}dx,
\end{align*}
then,
\begin{align}\label{3.5}
\tilde{\varphi}_i(l)=
\left\{\begin{array}{cc}>0 & \text{if}~ l<0 \\=0 & \text{if}~  l=0 \\< 0 & \text{if}~ l>0
\end{array}\right.,
\end{align}
which implies that $(ii)$ holds.
\end{proof}

Let $Q:=[\rho_1,R_1]\times[\rho_2,R_2]$, and let
\begin{equation*}
\gamma_0(t_1,t_2):=(t_1\star w_1,t_2\star w_2)\in H^{rad}_{a_1}\times H^{rad}_{a_2}, \quad\forall (t_1,t_2)\in Q.
\end{equation*}
We introduce the minimax class
\begin{equation*}
\Gamma:=\{\gamma\in C(Q, H^{rad}_{a_1}\times H^{rad}_{a_2}):\gamma=\gamma_0~\text{on}~\partial Q\}.
\end{equation*}

\begin{lemma}\label{l3.4}
There holds
\begin{equation*}
\sup_{\partial Q}E(\gamma_0)\leqslant
\max\{I_{\mu_1}(w_{a_1,\mu_1}),I_{\mu_2}(w_{a_2,\mu_2})\}+\varepsilon.
\end{equation*}
\end{lemma}
\begin{proof}
For every $(u_1,u_2)\in H^{rad}_{a_1}\times H^{rad}_{a_2}$, we have
\begin{align*}
E(u_1,u_2)=I_{\mu_1}(u_1)+I_{\mu_2}(u_2)-\frac{\beta}{p}\int_{\mathbb{R}}u^p_1u^p_2dx\leqslant I_{\mu_1}(u_1)+I_{\mu_2}(u_2).
\end{align*}
Then, from Lemma \ref{l3.3},
\begin{align*}
E(t_1\star w_1,\rho_2\star w_2)
&\leqslant I_{\mu_1}(t_1\star w_1)+I_{\mu_2}(\rho_2\star w_2)\nonumber\\
&\leqslant I_{\mu_1}(t_1\star w_1)+\varepsilon\nonumber\\
&\leqslant \sup_{l\in\mathbb{R}}I_{\mu_1}(l\star w_1)+\varepsilon.
\end{align*}
By Lemma \ref{l2.5}, we have
\begin{equation*}
w_{a_i,\mu_i}=\bar{l}_i\star w_i, \quad \text{for}~e^{\bar{l}_i}:=(\frac{\mu_i+\beta}{\mu_i})^\frac{1}{s[(p-1)N-2s]}.
\end{equation*}
Then, due to $l_1\star(l_2\star w)=(l_1+l_2)\star w$ for every $l_1,l_2\in\mathbb{R}$ and $w\in H^s(\mathbb{R})$, we have
\begin{equation*}
\sup_{l\in\mathbb{R}}I_{\mu_1}(l\star w_1)=\sup_{l\in\mathbb{R}}I_{\mu_1}(l\star w_{a_1,\mu_1}).
\end{equation*}
As a consequence of Lemma \ref{l2.7},
$$
\sup_{l\in\mathbb{R}}I_{\mu_1}(l\star w_{a_1,\mu_1})=I_{\mu_1}(w_{a_1,\mu_1}).
$$

Therefore, we have
\begin{equation*}
E(t_1\star w_1,\rho_2\star w_2)\leqslant I_{\mu_1}(w_{a_1,\mu_1})+\varepsilon,\quad \forall t_1\in [\rho_1,R_1].
\end{equation*}

Similarly, we have
\begin{equation*}
E(\rho_1\star w_1,t_2\star w_2)\leqslant I_{\mu_2}(w_{a_2,\mu_2})+\varepsilon,\quad \forall t_2\in [\rho_2,R_2],
\end{equation*}

\begin{align*}
E(t_1\star w_1,R_2\star w_2)
&\leqslant I_{\mu_1}(t_1\star w_1)+I_{\mu_2}(R_2\star w_2)\nonumber\\
&\leqslant \sup_{l\in\mathbb{R}} I_{\mu_1}(l\star w_1)=I_{\mu_1}(w_{a_1,\mu_1}),\quad \forall t_1\in [\rho_1,R_1],
\end{align*}
and
$$
E(R_1\star w_1,t_2\star w_2)\leqslant I_{\mu_2}(w_{a_2,\mu_2}),\quad \forall t_2\in [\rho_2,R_2].
$$
Hence, the conclusion of Lemma \ref{l3.4}  holds.
\end{proof}

\begin{lemma}\label{l3.5}
For every $\gamma\in \Gamma$, there exists $(t_{1,\gamma},t_{2,\gamma})\in Q$ such that $\gamma(t_{1,\gamma},t_{2,\gamma})\in\mathcal{P}(a_1,\mu_1+\beta)\times \mathcal{P}(a_2,\mu_2+\beta)$.
\end{lemma}
\begin{proof}
For $\gamma\in\Gamma$, we use the notation $\gamma(t_1,t_2)=(\gamma_1(t_1,t_2),\gamma_2(t_1,t_2))\in H^{rad}_{a_1}\times H^{rad}_{a_2}$. Considering the map $F_\gamma:Q\rightarrow \mathbb{R}^2$ defined by
\begin{equation*}
F_\gamma(t_1,t_2):=\left(\frac{\partial}{\partial l}I_{\mu_1+\beta}(l\star\gamma_1(t_1,t_2))|_{l=0},~\frac{\partial}{\partial l}I_{\mu_2+\beta}(l\star\gamma_2(t_1,t_2))|_{l=0}\right),
\end{equation*}
from
\begin{align*}
&\frac{\partial}{\partial l}I_{\mu_i+\beta}(l\star\gamma_i(t_1,t_2))|_{l=0}\nonumber\\
&=\frac{\partial}{\partial l}\left(\frac{e^{2s^2l}}{2}\|(-\Delta)^\frac{s}{2}\gamma_i(t_1,t_2)\|^2_{L^2}-\frac{e^{(p-1)Nsl}}{2p}(\mu_i+\beta)\|\gamma_i(t_1,t_2)\|^{2p}_{L^{2p}}\right)|_{l=0}\nonumber\\
&=s^2\|(-\Delta)^\frac{s}{2}\gamma_i(t_1,t_2)\|^2_{L^2}-\frac{(p-1)Ns}{2p}(\mu_i+\beta)\|\gamma_i(t_1,t_2)\|^{2p}_{L^{2p}},
\end{align*}
we deduce that
\begin{equation*}
F_\gamma(t_1,t_2)=(0,0)\quad\text{if and only if}\quad
\gamma(t_1,t_2)\in\mathcal{P}(a_1,\mu_1+\beta)\times \mathcal{P}(a_2,\mu_2+\beta).
\end{equation*}
Now, we will show that $F_\gamma(t_1,t_2)=(0,0)$ has a solution in $Q$ for every $\gamma\in\Gamma$. Since
\begin{align*}
F_{\gamma_0}(t_1,t_2)=&\left(s^2e^{2s^2t_1}\|(-\Delta)^\frac{s}{2} w_1\|^2_{L^2}-\frac{(p-1)Ns}{2p}e^{(p-1)Nst_1}(\mu_1+\beta)\|w_1\|^{2p}_{L^{2p}},\right.\nonumber\\
&\qquad \left.s^2e^{2s^2t_2}\|(-\Delta)^\frac{s}{2}w_2\|^2_{L^2}-\frac{(p-1)Ns}{2p}e^{(p-1)Nst_2}(\mu_2+\beta)\|w_2\|^{2p}_{L^{2p}}\right)\nonumber\\
=&(\tilde{\varphi}_1(t_1),\tilde{\varphi}_2(t_2)).
\end{align*}

By Lemma \ref{l3.3}, we get $(0,0)\notin F_{\gamma_0}(\partial Q)$, and $(0,0)$ is the only solution to $F_{\gamma_0}(t_1,t_2)=(0,0)$ in $Q$. It is easy to compute
$$
\deg(F_{\gamma_0},Q,(0,0))=sgn(\tilde{\varphi}'_1(0)\cdot \tilde{\varphi}'_2(0))=1.
$$

Now, for any $\gamma\in\Gamma$, since $F_\gamma(\partial^+Q)=F_{\gamma_0}(\partial^+Q)$, therefore, $(0,0)\notin F_{\gamma}(\partial Q)$, we get

\begin{equation*}
deg(F_\gamma,Q,(0,0))=deg(F_{\gamma_0},Q,(0,0))=1.
\end{equation*}
Hence, there exists a $(t_{1,\gamma},t_{2,\gamma})\in Q$ such that $F_{\gamma}(t_{1,\gamma},t_{2,\gamma})=(0,0)$.
\end{proof}

\begin{lemma}\label{l3.6}
There exists a bounded Palais-Smale sequence $(u_n,v_n)$ for $E$ on ${H}^{rad}_{a_1}\times{H}^{rad}_{a_2}$ at the level
\begin{equation}\label{eq3.8}
c:=\inf_{\gamma\in\Gamma} \max_{(t_1,t_2)\in Q} E(\gamma(t_1,t_2))>\max\{I_{\mu_1}(w_{a_1,\mu_1}), I_{\mu_2}(w_{a_2,\mu_2})\},
\end{equation}
satisfying the additional condition
\begin{equation}\label{3.11}
G(u_n,v_n)=o(1),
\end{equation}
where $o(1)\rightarrow 0$ as $n\rightarrow \infty$. Furthermore, there exists $\bar{C}>0$ such that
\begin{align*}
\int_{\mathbb{R}^N}(|(-\Delta)^\frac{s}{2}u_n|^2+|(-\Delta)^\frac{s}{2}v_n|^2)dx\geqslant \bar{C}\quad\text{for all}~n,
\end{align*}and $u_n^-,v_n^-\rightarrow 0$ a.e. in $\mathbb{R}^N$ as $n\rightarrow \infty$.
\end{lemma}
\begin{proof}
The idea comes from \cite{BJN18}. \eqref{eq3.8} is simply from Lemma \ref{l3.5}. We consider the augmented functional $\tilde{E}:\mathbb{R}\times H^{rad}_{a_1}\times H^{rad}_{a_2}\rightarrow\mathbb{R}$ defined by $\tilde{E}(l,u_1,u_2):=E(l\star u_1,l\star u_2)$.
Let
\begin{align*}
\tilde{\gamma}(t_1,t_2):=(l(t_1,t_2),\gamma_1(t_1,t_2),\gamma_2(t_1,t_2)),
\end{align*}
\begin{align*}
\tilde{\gamma}_0(t_1,t_2):=(0,\gamma_0(t_1,t_2))=(0,t_1\star w_1,t_2\star w_2),
\end{align*}
\begin{align*}
\tilde{\Gamma}:=\{\tilde{\gamma}\in C(Q,\mathbb{R}\times H^{rad}_{a_1}\times H^{rad}_{a_2}:\tilde{\gamma}=\tilde{\gamma}_0~\text{on}~\partial Q)\},
\end{align*}
and
\begin{align*}
\tilde{c}:=\inf_{\tilde{\gamma}\in\tilde{\Gamma}} \max_{(t_1,t_2)\in Q} \tilde{E}(\tilde{\gamma}(t_1,t_2)).
\end{align*}
Since for any $\gamma(t_1,t_2)=(\gamma_1(t_1,t_2),\gamma_2(t_1,t_2))\in\Gamma$, $(0,\gamma_1(t_1,t_2),\gamma_2(t_1,t_2))\in \tilde{\Gamma}$, we have $\tilde{c}\leqslant c$. On the other hand, for any $\tilde{\gamma}\in \tilde{\Gamma}$ and $(t_1,t_2)\in Q$, we have
\begin{align*}
\tilde{E}(\tilde{\gamma}(t_1,t_2))=E(l(t_1,t_2)\star \gamma_1(t_1,t_2),l(t_1,t_2)\star\gamma_2(t_1,t_2)),
\end{align*}
and $(l(\cdot)\star\gamma_1(\cdot),l(\cdot)\star\gamma_2(\cdot))\in\Gamma$ due to $\tilde{\gamma}=\tilde{\gamma}_0~\text{on}~\partial Q$, so $c\leqslant\tilde{c}$. Hence,  $c=\tilde{c}$.

Now take a sequence of $\{\tilde{\gamma}_n\}\subset \tilde{\Gamma}$ such that
$$
\lim\limits_{n\to+\infty}\max\limits_{(t_1,t_2)\in Q}\tilde{E}(\tilde{\gamma}_n(t_1,t_2))=\tilde{c}=c.
$$

We may also assume that $\tilde{\gamma}_n=(l_n,\gamma_{1,n},\gamma_{2,n})$ satisfies the following two additional properties: for all $(t_1,t_2)\in Q$,
\begin{itemize}
\item $l_n(t_1,t_2)\equiv 0$,
\item $\gamma_{1,n}(t_1,t_2)\ge 0,\quad \gamma_{2,n}(t_1,t_2)\ge 0,\quad \text{ a.e. in }\mathbb{R}^N$.
\end{itemize}

The first property comes from the fact that
\begin{align*}
\tilde{E}(\tilde{\gamma}(t_1,t_2))
&=E(l(t_1,t_2)\star \gamma_1(t_1,t_2),l(t_1,t_2)\star\gamma_2(t_1,t_2))\nonumber\\
&=\tilde{E}(0,l(t_1,t_2)\star \gamma_1(t_1,t_2),l(t_1,t_2)\star\gamma_2(t_1,t_2)),
\end{align*}
and the second one is the consequence of $\tilde{E}(l,|u|,|v|)\le\tilde{E}(l,u,v)$ and the definition of $\tilde{c}$.

Applying Theorem 3.2 in \cite{G93},  there exists a Palais-Smale sequence $(l_n,u_n,v_n)$ for $\tilde{E}$ on $\mathbb{R}\times H^{rad}_{a_1}\times H^{rad}_{a_2}$ at level $\tilde{c}$, such that
\begin{itemize}
\item $\lim\limits_{n\to+\infty}\tilde{E}(l_n,u_n,v_n)=\tilde{c}=c$,
\item $\lim\limits_{n\to+\infty}|l_n|+dist((u_n,v_n),\tilde{\gamma}_n(Q))=0$,
\item For all $u,v\in H_r^s(\mathbb R^N)$ with $\int_{\mathbb R^N}u_nudx=0$, $\int_{\mathbb R^N}v_nvdx=0$ and $\forall ~l\in \mathbb R$,
$$
\langle \tilde{E}'(l_n,u_n,v_n),(l,u,v)\rangle=o(1)(|l|+\|u\|_{H^s}+\|v\|_{H^s}).
$$
\end{itemize}

Taking $(l,u,v)=(1,0,0)$, direct calculations gives that
\begin{align}\label{eq3.6}
\langle \tilde{E}'&(l_n,u_n,v_n),(1,0,0)\rangle=s^2e^{2s^2l_n}\int_{\mathbb R^N}(|(-\Delta)^\frac{s}{2}u_n|^2+|(-\Delta)^\frac{s}{2}v_n|^2)dx\\ \nonumber
&-\frac{e^{(p-1)Nsl_n}(p-1)Ns}{2p}\int_{\mathbb{R}^N}(\mu_1|u_n|^{2p}+2\beta |u_n|^p|v_n|^p+\mu_2|v_n|^{2p})dx.
\end{align}

From above, we can get that
\begin{align*}
&se^{2s^2l_n}(\frac{(p-1)N}{2}-s)\int_{\mathbb R^N}(|(-\Delta)^\frac{s}{2}u_n|^2+|(-\Delta)^\frac{s}{2}v_n|^2)dx\\
&=(p-1)Ns\tilde{E}(l_n,u_n,v_n)-\langle \tilde{E}'(l_n,u_n,v_n),(1,0,0)\rangle\\ \nonumber
&\to (p-1)Nsc,\quad \text{as}~ n\to+\infty.
\end{align*}

Since $l_n\to 0$ and $p>1+\frac{2s}{N}$, we get that there exist $\bar{C}>0$ and $C>0$, such that
\begin{equation}\label{eq3.7}
\bar{C}\le \int_{\mathbb R^N}(|(-\Delta)^\frac{s}{2}u_n|^2+|(-\Delta)^\frac{s}{2}v_n|^2)dx\le C,
\end{equation}
therefore $(u_n,v_n)$ is bounded in $H^{s}_r(\mathbb R^N)\times H^{s}_r(\mathbb R^N)$.  Using $l_n\to 0$ and \eqref{eq3.6} again, we conclude that $(u_n,v_n)$ satisfies \eqref{3.11}. Now taking $(l,u,v)=(0,u,v)$ for any $(u,v)\in H^{s}_r(\mathbb R^N)\times H^{s}_r(\mathbb R^N)$ with $\int_{\mathbb R^N}u_nudx=0$, $\int_{\mathbb R^N}v_nvdx=0$, due to boundedness of $(u_n,v_n)$ and $l_n\to 0$, it is easy to see that
\begin{align*}
\langle E'(u_n,v_n),(u,v)\rangle&=\langle \tilde{E}'(l_n,u_n,v_n),(0,u,v)\rangle
+O(|l_n|)(\|u\|_{H^s}+\|v\|_{H^s})\\
&=o(1)(\|u\|_{H^s}+\|v\|_{H^s}).
\end{align*}

Therefore, $(u_n,v_n)$ is a bounded Palais-Smale sequence for $E$ on $H^{rad}_{a_1}\times H^{rad}_{a_2}$ at level $c$ with additional condition \eqref{3.11}. Finally, $u_n^-,v_n^-\rightarrow 0$ a.e. in $\mathbb{R}^N$ as $n\rightarrow \infty$ is a simple consequence of $\gamma_{1,n}(t_1,t_2)\ge 0,\quad \gamma_{2,n}(t_1,t_2)\ge 0$ and $\lim\limits_{n\to+\infty} dist((u_n,v_n),\tilde{\gamma}_n(Q))=0$.
\end{proof}

From Lemma \ref{l3.6}, there exists nonnegative functions $\tilde{u},~\tilde{v}$ in $H_r^s(\mathbb{R}^N)$, such that, up to a subsequence
\begin{align*}
&(u_n,v_n)\rightharpoonup (\tilde{u},\tilde{v}),\quad\text{weakly in} \quad  H^s(\mathbb{R}^N)\times H^s(\mathbb{R}^N),\\
&(u_n,v_n)\rightarrow (\tilde{u},\tilde{v}),\quad\text{strongly in}\quad  L^{2p}(\mathbb{R}^N)\times L^{2p}(\mathbb{R}^N),\\
&(u_n,v_n)\rightarrow (\tilde{u},\tilde{v}),\quad a.e. \text{~in}~ \mathbb{R}^N.
\end{align*}

As a consequence $E'|_{H^{rad}_{a_1}\times H^{rad}_{a_2}}(u_n,v_n)\rightarrow 0$, there exist two sequences of real number $\{\lambda_{1,n}\}$ and $\{\lambda_{2,n}\}$ such that
\begin{align}\label{o}
\begin{split}
&\int_{\mathbb{R}^N}\big((-\Delta)^{\frac{s}{2}}u_n(-\Delta)^{\frac{s}{2}}g+(-\Delta)^{\frac{s}{2}}v_n(-\Delta)^{\frac{s}{2}}h-\mu_1|u_n|^{2p-2}u_ng-\mu_2|v_n|^{2p-2}v_nh\big)dx\\
&-\int_{\mathbb R^n}(\beta|u_n|^{p-2}|v_n|^pu_ng+|u_n|^p|v_n|^{p-2}v_nh)dx+\int_{\mathbb{R}^N}(\lambda_{1,n}u_ng+\lambda_{2,n}v_nh)dx\\
&=o(1)(\|g\|_{H^s}+\|h\|_{H^s}),
\end{split}
\end{align}
for every $g,~h\in H^s({\mathbb{R}^N})$ with $o(1)\rightarrow0,~\text{as}~n\rightarrow \infty $.

\begin{lemma}\label{l3.8}
Both $\{\lambda_{1,n}\}$ and $\{\lambda_{2,n}\}$ are bounded sequences and at least one of them is converging, up to a sequence, to a positive value.
\end{lemma}
\begin {proof}
By using $(u_n,0)~\text{and}~(0,v_n)$ as test functions in \eqref{o}, we get:
\begin{align*}
&\int_{\mathbb{R}^N}\big(|(-\Delta)^\frac{s}{2}u_n|^2-\mu_1|u_n|^{2p}-\beta |u_n|^p|v_n|^p\big)dx+\lambda_{1,n}a_1^2=o(1),\\
&\int_{\mathbb{R}^N}\big(|(-\Delta)^\frac{s}{2}v_n|^2-\mu_2|v_n|^{2p}-\beta |u_n|^p|v_n|^p\big)dx+\lambda_{2,n}a_2^2=o(1),
\end{align*}
with $o(1)\rightarrow 0,~\text{as}~n\rightarrow\infty$. Hence the boundedness of $\{\lambda_{i,n}\}$ follows from the boundedness of $u_n,~v_n$ in $H^s({\mathbb{R}^N})$ and in $L^{2p}({\mathbb{R}^N})$.
Furthermore, since $(u_n,v_n)$ satisfies \eqref{3.11}, there holds
\begin{align*}
{\lambda_{1,n}}{a_1^2}+{\lambda_{2,n}}{a_2^2}=&-\int_{\mathbb{R}^N}\bigg({|(-\Delta)^{\frac{s}{2}}u_n|^2+|(-\Delta)^{\frac{s}{2}}v_n|^2}-\mu_1{u_n^{2p}}-2\beta {u_n^p}{v_n^p}-\mu_2{v_n^{2p}}\bigg)dx+o(1)\\
&=(\frac{2ps}{(p-1)N}-1)\int_{\mathbb{R}^N}\bigg(|(-\Delta)^{\frac{s}{2}}u_n|^2+|(-\Delta)^{\frac{s}{2}}v_n|^2\bigg)dx+o(1),
\end{align*}
therefore by \eqref{eq3.7},
\begin{equation*}
(\frac{ps}{(p-1)N}-\frac{1}{2})\bar{C}\le {\lambda_{1,n}}{a_1^2}+{\lambda_{2,n}}{a_2^2}\le 2(\frac{2ps}{(p-1)N}-1)C,
\end{equation*}
for $1+\frac{2s}{N}<p<\frac{N}{N-2s}$ and every $n$ sufficiently large. Therefore, at least one sequence of $\{\lambda_{i,n}\}$ is positive and bounded away from $0$. This shows that at least one sequence of $\{\lambda_{i,n}\}$ is converging, up to a sequence, to a positive value.
\end{proof}

Next, we consider converging subsequence $\lambda_{1,n}\rightarrow\tilde{\lambda}_1\in {\mathbb{R}}$ and $\lambda_{2,n}\rightarrow\tilde{\lambda}_2\in {\mathbb{R}},~\text{as}~n\rightarrow\infty$. The sign of $\tilde{\lambda}_i$ plays an important role for the strong convergence of $u_n,~v_n$ in $H^s(\mathbb{R}^N)$.

\begin{lemma}\label{l3.9}
If $\tilde{\lambda}_1>0$ (resp. $\tilde{\lambda}_2>0$), then $u_n\rightarrow \tilde{u}$ (resp. $v_n\rightarrow \tilde{v}$) strongly in $H^s(\mathbb{R}^N)$.
\end{lemma}
\begin{proof}
Let us suppose that $\tilde{\lambda}_1>0$. By weak convergence of $u_n$ in $H^s(\mathbb{R}^N)$ and strongly convergence in $L^{2p}(\mathbb{R}^N)$, it is easy to get from \eqref{o}:
\begin{align*}
o(1)&=\langle E'({u_n},{v_n})-E'(\tilde{u},\tilde{v}),({u_n}-\tilde{u},0)\rangle +{\tilde{\lambda}_1}\int_{\mathbb{R}^N}({u_n}-\tilde{u})^2dx\\
&=\int_{\mathbb{R}^N}\bigg(|(-\Delta)^{\frac{s}{2}}({u_n}-\tilde{u})|^2+\tilde{\lambda}_1({u_n}-\tilde{u})^2\bigg)dx+o(1),
\end{align*}
with $o(1)\rightarrow0$ and $n\rightarrow\infty$. Since $\tilde{\lambda}_1>0$, this is equivalent to the strong convergence of $u_n$ in $H^s(\mathbb R^N)$. The proof in the case $\tilde{\lambda}_2>0$ is similar.
\end{proof}

At the end of this section, we give the proof of Theorem \ref{t1.1}.
\begin{proof}[{\bf Proof of Theorem \ref{t1.1}}]
By the convergence of $\{\lambda_{1,n}\}$ and $\{\lambda_{2,n}\}$, and the weak convergence $(u_n,v_n)\rightharpoonup(\tilde{u},\tilde{v})$, we obtain that $(\tilde{\lambda}_1,\tilde{\lambda}_2,\tilde{u},\tilde{v})$ is a solution of \eqref{GPE} with at least one $\tilde{\lambda}_i$ positive. We will show that both $\tilde{\lambda}_1, \tilde{\lambda}_2$ are positive, hence by Lemma \ref{l3.9}, $\tilde{u}\in H_{a_1}$, $\tilde{v}\in H_{a_2}$ and the proof is complete.

We prove by contradiction. Without loss of generality, by Lemma \ref{l3.9}, we may assume that $\tilde{\lambda}_1>0$ and $\tilde{\lambda}_2\le 0$. Since $(\tilde{\lambda}_1,\tilde{\lambda}_2,\tilde{u},\tilde{v})$ is a solution of \eqref{GPE} and $\tilde{u},\tilde{v}\geq0$, we have that
\begin{align*}
(-\Delta)^s \tilde v=-\tilde{\lambda}_2 \tilde v +\mu_2 {\tilde v}^{2p-1}+\beta {\tilde u}^p{\tilde v}^{p-1}\geq0  \quad\text{in}~{\mathbb{R}^N},
\end{align*}
and since $2s<N\le 4s$, i.e., $2\le \frac{N}{N-2s}$, from Lemma \ref{l2.3} $(i)$, we can deduce that $\tilde{v}\equiv0$. In particular, this implies that $\tilde{u}$ solves
\begin{equation}
\left\{\begin{aligned}
&(-\Delta)^s{\tilde u}+{\tilde{\lambda}_1}{\tilde u} -{\mu_1}{\tilde u}^{2p-1}=0~~~~\text{in}~{\mathbb{R}^N},\\
&\int_{\mathbb{R}^N} \tilde u^2dx=a_1^2,\quad\text{and} ~\tilde u>0  \quad\text{in}~{\mathbb{R}^N},
\end{aligned}\right.
\end{equation}
so that $\tilde u=w_{a_1,\mu_1}\in \mathcal{P}(a_1,\mu_1)$. However, due to strong convergence of $u_n,~v_n$ in $L^{2p}(\mathbb R^N)$, we obtain due to \eqref{3.11},
\begin{align*}
c=\lim_{n\rightarrow\infty}E(u_n,v_n)&=\lim_{n\rightarrow\infty}{\frac{(p-1)N-2s}{4ps}}\int_{\mathbb{R}^N}(\mu_1u_n^{2p}+2\beta u_n^pv_n^p+\mu_2 v_n^{2p})dx\\
&=\frac{(p-1)N-2s}{4ps}\int_{\mathbb{R}^N}\mu_1{w_{a_1,\mu_1}}^{2p}dx=I_{\mu_1}(w_{a_1,\mu_1}).
\end{align*}
This is a contradiction with Lemma \ref{l3.6}.  Therefore, both $\tilde{\lambda}_1, \tilde{\lambda}_2$ are positive.
\end{proof}

\section{Proof of Theorem \ref{t1.2}}\label{sec4}

In this section, we devote to prove Theorem \ref{t1.2}. The proof is divided into two parts. Firstly, we show the existence of a positive solution $(\bar u,\bar v)$, and secondly we characterize it as a ground state. The proof of the theorem is based on a mountain pass argument. For $(u,v)\in H^{rad}_{a_1}\times H^{rad}_{a_2}$, we consider the function
\begin{equation*}
E(l\star(u,v))=\frac{e^{2s^2l}}{2} \int_{\mathbb{R}^N}(|(-\Delta)^\frac{s}{2}u|^2+|(-\Delta)^\frac{s}{2}v|^2)dx-\frac{e^{(p-1)Nsl}}{2p}\int_{\mathbb{R}^N}(\mu_1u^{2p}+2\beta u^pv^p+\mu_2{v^{2p}})dx,
\end{equation*}
where $l\star(u,v)=(l\star{u},l\star{v})$. If $(u,v)\in{H^{rad}_{a_1}}\times{H^{rad}_{a_2}}$, then ~ $l\star(u,v)\in{H^{rad}_{a_1}}\times{H^{rad}_{a_2}}$~ for any $l\in\mathbb{R}$. Similar to Lemma \ref{l2.7}, it holds:

\begin{lemma}\label{l3.11}
Let $(u,v)\in{H^{rad}_{a_1}\times{H^{rad}_{a_2}}}$. Then
\begin{align}
&\lim_{l\rightarrow{-\infty}}{\int_{\mathbb{R}}(|(-\Delta)^\frac{s}{2}{l\star{u}}|^2+|(-\Delta)^\frac{s}{2}{l\star{v}}|^2)dx=0},\\
&\lim_{l\rightarrow{+\infty}}{\int_{\mathbb{R}}(|(-\Delta)^\frac{s}{2}{l\star{u}}|^2+|(-\Delta)^\frac{s}{2}{l\star{v}}|^2)dx={+\infty}},\\
&\lim_{l\rightarrow{-\infty}}{E(l\star(u,v))=0^{+}}\text{and}~\lim_{l\rightarrow{+\infty}}{E(l\star(u,v))={-\infty}}.
\end{align}
\end{lemma}

The next lemma enlightens the mountain pass structure of the problem.

\begin{lemma}\label{l3.12}
There exists $K>0$ sufficiently small such that
\begin{align}
\sup_{A}E<\inf_{B}E~\text{and}~E(u,v)>0~\text{on}~A,
\end{align}
where
\begin{align}\label{eq4.2}
&A=\{(u,v)\in{H^{rad}_{a_1}}\times{H^{rad}_{a_2}}, |(-\Delta)^{\frac{s}{2}}u|^2_{L^2}+|(-\Delta)^{\frac{s}{2}}v|^2_{L^2}\leq{K}\},\\
&B=\{(u,v)\in{H^{rad}_{a_1}}\times{H^{rad}_{a_2}}, |(-\Delta)^{\frac{s}{2}}u|^2_{L^2}+|(-\Delta)^{\frac{s}{2}}v|^2_{L^2}=2K\}.
\end{align}
\end{lemma}
\begin{proof}
By Gagliardo-Nirenberg-Sobolev inequality \eqref{Gagliardo-Nirenberg-Sobolev inequality}, there holds
\begin{align*}
&\int_{\mathbb{R}^N}(\mu_1u^{2p}+2\beta u^pv^p+\mu_2 v^{2p})dx\\
&\leq C \int_{\mathbb{R}^N}( u^{2p}+ v^{2p})dx\\
&\leq C\bigg(\int_{\mathbb{R}^N}(|(-\Delta)^\frac{s}{2}u|^2+|(-\Delta)^\frac{s}{2}v|^2)dx\bigg)^{\frac{(p-1)N}{2s}},
\end{align*}
for every $(u,v)\in{H^{rad}_{a_1}}\times{H^{rad}_{a_2}}$, where $C>0$ depends on $\mu_1,\mu_2,\beta,a_1,a_2>0$, but not on the choice of $(u,v)$. Now if $(u_1,v_1)\in B$ and $(u_2,v_2)\in A$ (with $K$ to be determined), we have
\begin{align*}
&E(u_1,v_1)-E(u_2,v_2)\\
&\geq{\frac{1} {2}}\Bigg\{ \int_{\mathbb{R}^N}(|(-\Delta)^{\frac{s}{2}}{u_1}|^2+|(-\Delta)^{\frac{s}{2}}{v_1}|^2)dx-\int_{\mathbb{R}^N}(|(-\Delta)^{\frac{s}{2}}{u_2}|^2+|(-\Delta)^{\frac{s}{2}}{v_2}|^2)dx\Bigg\}\\
&~~-\frac{1}{2p}\int_{\mathbb{R}^N}\Bigg(\mu_1u_1^{2p}+2\beta{u_1^pv_1^p}+\mu_2 v_1^{2p}\Bigg)dx\geq \frac{K}{2}-\frac{C}{2p}(2K)^{\frac{(p-1)N}{2s}}\ge \frac{K}{4},
\end{align*}
provided $K>0$ is sufficiently small. Furthermore if necessary, we can make $K$ smaller, then there holds
\begin{align*}
E(u_2,v_2)&\geq \frac{1}{2}\Bigg(\int_{\mathbb{R}^N}(|(-\Delta)^\frac{s}{2}u_2|^2+|(-\Delta)^\frac{s}{2}v_2|^2)dx\Bigg)\\
&-\frac{C}{2p}\Bigg(\int_{\mathbb{R}^N}(|(-\Delta)^\frac{s}{2}u_2|^2+|(-\Delta)^\frac{s}{2}v_2|^2)dx\Bigg)^{\frac{(p-1)N}{2s}}>0,
\end{align*}
for every $(u_2,v_2)\in A$.
\end{proof}

For the next part, we shall introduce a suitable minimax class. Define
\begin{equation}\label{eq4.3}
D:=\{(u,v)\in{H^{rad}_{a_1}}\times{H^{rad}_{a_2}}:|(-\Delta)^{\frac{s}{2}}u|^2_{L^2}+|(-\Delta)^{\frac{s}{2}}v|^2_{L^2}\geq 3K~\text{and}~E(u,v)\leq0\}.
\end{equation}

Recall from Lemma \ref{l2.5} that $w_{a,\mu}$ is the unique positive radial solution of \eqref{2.5}. By Lemma \ref{l3.11}, there exist $l_1<0$ and $l_2>0$ such that
\begin{align*}
&l_1\star(w_{a_1,{C_0}/{a_1^2}},w_{a_2,{C_0}/{a_2^2}})=:(\bar u_1,\bar v_1)\in A,\\
&l_2\star(w_{a_1,{C_0}/{a_1^2}},w_{a_2,{C_0}/{a_2^2}})=:(\bar u_2,\bar v_2)\in D.
\end{align*}
At last, we define
\begin{equation*}
\bar{\Gamma}:=\{\bar{\gamma}\in{C([0,1],{H^{rad}_{a_1}}\times{H^{rad}_{a_2}})}:\bar{\gamma}(0)=(\bar{u}_1,\bar{v}_1),\bar{\gamma}(1)=(\bar{u}_2,\bar{v}_2)\}.
\end{equation*}

Similarly to the proof of Lemma \ref{l3.6}, we can derive the following lemma.
\begin{lemma}\label{l3.13}
There exists a bounded Palais-Smale sequence $(u_n,v_n)$ for $E$ on $H^{rad}_{a_1}\times H^{rad}_{a_2}$ at the level
\begin{equation*}
d:=\inf_{\bar\gamma\in\bar\Gamma}\max_{t\in[0,1]}E(\bar\gamma(t)),
\end{equation*}
satisfying the additional condition
\begin{equation*}
G(u_n,v_n)=o(1),
\end{equation*}
with $o(1)\rightarrow0$ as $n\rightarrow\infty$. Furthermore, $u^{-}_n,v^{-}_n\rightarrow0$ a.e. in $\mathbb{R}^N$ as $n\rightarrow\infty$.
\end{lemma}

\begin{lemma}\label{l3.14}
Let $\beta_2$ be defined in \eqref{eq1.2}, if $\beta>\beta_2$, then
\begin{equation}\label{eq4.1}
\sup_{l\in\mathbb{R}}E\Bigg(l\star (w_{a_1,({C_0}/{a^2_1})^{p-1}},w_{a_2,({C_0}/{a^2_2})^{p-1}})\Bigg)<\min\Bigg\{I_{\mu_1}(w_{a_1,\mu_1}),I_{\mu_2}(w_{a_2,\mu_2})\Bigg\}.
\end{equation}
\end{lemma}
\begin{proof}
By Lemma \ref{l2.5}, direct computation gives
\begin{align*}
&\int_{\mathbb{R}^N}\Bigg(l\star w_{a_1,({C_0}/{a^2_1})^{p-1}}\Bigg)^p\Bigg(l\star w_{a_2,({C_0}/{a^2_2})^{p-1}}\Bigg)^pdx\\
&=\int_{\mathbb{R}^N}e^{(p-1)Nsl}\Bigg(\frac{{a_1}}{C_0^\frac{1}{2}}{w_0(x)}\Bigg)^p\Bigg(\frac{{a_2}}{C_0^\frac{1}{2}}{w_0(x)}\Bigg)^pdx\\
&=\frac{a_1^pa_2^p}{C_0^p}e^{(p-1)Nsl}\int_{\mathbb{R}^N}w_0^{2p}dx=\frac{a_1^pa_2^pC_1}{C_0^p}e^{(p-1)Nsl}.
\end{align*}
Furthermore,
\begin{align*}
E\Bigg(l\star (w_{a_1,({C_0}/{a^2_1})^{p-1}},w_{a_2,({C_0}/{a^2_2})^{p-1}})\Bigg)=&\frac{(p-1)Ne^{2s^2l}}{4ps}(\frac{a_1^2C_1+a_2^2C_1}{C_0})
\\&-\frac{e^{(p-1)Nsl}}{2p}(\frac{\mu_1C_1a_1^{2p}+2\beta C_1 a_1^pa_2^p+\mu_2C_1a_2^{2p}}{C_0^p}).
\end{align*}

Therefore, it is easy to get that
\begin{align*}
&\max_{l\in \mathbb{R}}E\Bigg(l\star (w_{a_1,({C_0}/{a^2})^{p-1}},w_{a_2,({C_0}/{a^2})^{p-1}})\Bigg)\\
&=\frac{(p-1)N-2s}{4ps}\frac{C_1C_0^{\frac{2ps-(p-1)N}{(p-1)N-2s}}(a_1^2+a_2^2)^{\frac{(p-1)N}{(p-1)N-2s}}}{(\mu_1a_1^{2p}+2\beta a_1^pa_2^p+\mu_2a_2^{2p})^{\frac{2s}{(p-1)N-2s}}}.
\end{align*}

Due to \eqref{eq1.2} and \eqref{2.29},  if $\beta>\beta_2$, \eqref{eq4.1} is satisfied.
\end{proof}

\noindent\textbf{Existence of a positive solution at level $d$.}\\
\indent We will prove the existence of positive solution at level $d$ by contradiction. By Lemma \ref{l3.13}, up to a subsequence, we may assume that
\begin{align*}
&(u_n,v_n)\rightharpoonup (\bar{u},\bar{v}),\quad\text{weakly in} \quad  H^s(\mathbb{R}^N)\times H^s(\mathbb{R}^N),\\
&(u_n,v_n)\rightarrow (\bar{u},\bar{v}),\quad\text{strongly in}\quad  L^{2p}(\mathbb{R}^N)\times L^{2p}(\mathbb{R}^N),\\
&(u_n,v_n)\rightarrow (\bar{u},\bar{v}),\quad a.e. \text{~in}~ \mathbb{R}^N.
\end{align*}

Then, it can be easily derived that $(\bar u,\bar v)$ is a solution of \eqref{GPE} for some constants $\bar\lambda_1,\bar\lambda_2\in\mathbb{R}$. Moreover, Lemma \ref{l3.8} and Lemma \ref{l3.9} are applicable. We may assume that $\bar\lambda_1>0$ and $u_n\to \bar{u}$ strongly in $H^s(\mathbb R^N)$.  If $\bar\lambda_2\le 0$, we can derive that $\bar{v}\equiv0$ and $\bar{u}=w_{a_1,\mu_1}$ as in the proof of Theorem \ref{t1.1}. By $G(u_n,v_n)\to 0$ and strong convergence in $L^{2p}(\mathbb R^N)$, $d=I_{\mu_1}(w_{a_1,\mu_1})$. We can consider the path
\begin{equation*}
\bar\gamma(t):=((1-t)l_1+tl_2)\star(w_{a_1,\mu_1},w_{a_2,\mu_2}).
\end{equation*}
Obviously, $\bar\gamma\in\bar\Gamma$. Then by Lemma \ref{l3.14},
\begin{equation*}
d\leq\sup_{t\in[0,1]}E(\bar\gamma(t))\leq\sup_{l\in\mathbb{R}}E(l\star(w_{a_1,\mu_1},w_{a_2,\mu_2}))<I_{\mu_1}(w_{a_1,\mu_1}),
\end{equation*}
which is a contradiction. Therefore, $\bar\lambda_2>0$ and $v_n\to \bar v$ strongly in $H^s(\mathbb R^N)$. This gives that $(\bar\lambda_1,\bar\lambda_2,\bar u,\bar v)$ is a solution of
\eqref{GPE} with $\bar\lambda_1,\bar\lambda_2>0$ and $(\bar u,\bar v)\in H_{a_1}\times H_{a_2}$.

Obviously, we can see that $G(\bar u,\bar v)=0$, i.e., $(\bar u,\bar v)\in F$.

\textbf{Variational characterization of $(\bar{u},\bar{v})$.} In the following, we will prove that
\begin{align*}
E(\bar u,\bar v)&=\inf\{E(u,v);(u,v)\in F\}=\inf_{(u,v)\in H_{a_1}\times H_{a_2}}\mathcal{R}(u,v),
\end{align*}
where $F$ and $\mathcal{R}$ are defined in \eqref{eq1.3} and \eqref{eq1.4}.  Recall the definition of $A$ in \eqref{eq4.2} and $D$ in \eqref{eq4.3}, let us define
\begin{align*}
&A^{+}:=\{(u,v)\in A,~u,~v\geq0~\text{a.e.~in}~\mathbb{R}^N\},\\
&D^{+}:=\{(u,v)\in D,~u,~v\geq0~\text{a.e.~in}~\mathbb{R}^N\}.
\end{align*}
For any $(u_1,v_1)\in A^{+}$ and $(u_2,v_2)\in D^{+}$, let
\begin{align*}
\bar\Gamma(u_1,v_1,u_2,v_2)=:\Bigg\{\bar\gamma\in C([0,1],{H^{rad}_{a_1}}\times{H^{rad}_{a_2}}):\bar\gamma(0)=(u_1,v_1),\bar\gamma(1)=(u_2,v_2)\Bigg\}.
\end{align*}

\begin{lemma}\label{l3.15}
The sets $A^{+}$ and $D^{+}$ are connected by arcs, so that
\begin{equation}\label{3.19}
d=\inf_{\bar\gamma\in\bar\Gamma(u_1,v_1,u_2,v_2)}\max_{t\in[0,1]}E(\bar\gamma(t)),
\end{equation}
for every $(u_1,v_1)\in A^{+}$ and $(u_2,v_2)\in D^{+}$.
\end{lemma}
\begin{proof}
Equality \eqref{3.19} follows easily once we show that $A^{+}$ and $D^{+}$ are connected by arcs (~as $l*(\bar{u},\bar{v})$ is a path from $A^+$ to $D^{+}$ ).
Let $(u_1,v_1),(u_2,v_2)\in{H^{rad}_{a_1}}\times{H^{rad}_{a_2}}$ be nonnegative functions such that
\begin{align}\label{3.24}
\int_{\mathbb{R}^N}(|(-\Delta)^\frac{s}{2}{u_1}|^2+|(-\Delta)^\frac{s}{2}{v_1}|^2)dx=\int_{\mathbb{R}^N}(|(-\Delta)^\frac{s}{2}u_2|^2+|(-\Delta)^\frac{s}{2}v_2|^2)dx=\alpha^2,
\end{align}
for some $\alpha>0$.
For $l\in\mathbb{R}$ and $\theta\in[0,\frac{\pi}{2}]$,
\begin{equation*}
h(l,\theta)=(\cos\theta(l\star u_1)(x)+\sin\theta(l\star u_2)(x),\cos\theta(l\star v_1)(x)+\sin\theta(l\star v_2)(x)).
\end{equation*}
Set $h=(h_1,h_2)$, we have that $h_1(l,\theta),~h_2(l,\theta)\geq0$ a.e. in $\mathbb{R}^N$.
It is not difficult to check that
\begin{align*}
&\int_{\mathbb{R}^N}{h^2_1(l,\theta)}dx=a^2_1+\sin(2\theta)\int_{\mathbb{R}^N}{u_1}{u_2}dx,\quad \int_{\mathbb{R}^N}{h^2_2(l,\theta)}dx=a^2_2+\sin(2\theta)\int_{\mathbb{R}^N}{v_1}{v_2}dx,\\
&\int_{\mathbb{R}^N}(|(-\Delta)^\frac{s}{2}{h_1(l,\theta)}|^2+|(-\Delta)^\frac{s}{2}{h_2(l,\theta)}|^2)dx\\
&=e^{2s^2l}\bigg(\alpha^2+\sin(2\theta)\int_{\mathbb{R}^N}((-\Delta)^\frac{s}{2}{u_1}(-\Delta)^\frac{s}{2}{u_2}+(-\Delta)^\frac{s}{2}{v_1}(-\Delta)^\frac{s}{2}{v_2})dx\bigg),
\end{align*}
for all $(l,\theta)\in\mathbb{R}\times [0,\frac{\pi}{2}]$. We can deduce that
\begin{align*}
&a_1^2\leq\int_{\mathbb{R}^N}h^2_1(l,\theta)dx\leq2a_1^2\quad\text{and}\quad a_2^2\leq\int_{\mathbb{R}^N}h^2_2(l,\theta)dx\leq2a_2^2.
\end{align*}

Therefore
\begin{equation*}
\alpha^2+\sin(2\theta)\int_{\mathbb{R}^N}((-\Delta)^\frac{s}{2}{u_1}(-\Delta)^\frac{s}{2}{u_2}+(-\Delta)^\frac{s}{2}{v_1}(-\Delta)^\frac{s}{2}{v_2})dx>0,
\end{equation*}and is continuous in $\theta\in [0,\frac{\pi}{2}]$, so there is a constant $C>0$ independent of $l,\theta$, such that
\begin{equation*}
C\alpha^2e^{2s^2l}\leq\int_{\mathbb{R}^N}(|(-\Delta)^\frac{s}{2}h_1(l,\theta)|^2+|(-\Delta)^\frac{s}{2}h_2(l,\theta)|^2)dx\leq2\alpha^2e^{2s^2l}.
\end{equation*}
Thus we can define the function
\begin{equation*}
\hat {h}(l,\theta)(x)=(a_1\frac{h_1(l,\theta)}{\|h_1(l,\theta)\|_{L^2}},a_2\frac{h_2(l,\theta)}{\|h_2(l,\theta)\|_{L^2}}),
\end{equation*}
for $(l,\theta)\in\mathbb{R}\times[0,\frac{\pi}{2}]$.\\
Notice that $\hat{h}(l,\theta)\in{H^{rad}_{a_1}}\times{H^{rad}_{a_2}}$ for every $(l,\theta)$, we can obtain that
\begin{align}\label{3.25}
C c_0\alpha^2e^{2s^2l}\leq\int_{\mathbb{R}^N}(|(-\Delta)^\frac{s}{2}\hat{h_1}(l,\theta)|^2+|(-\Delta)^\frac{s}{2}\hat{h_2}(l,\theta)|^2)dx\leq2 c_1\alpha^2e^{2s^2l},
\end{align}with $c_0=\frac{\min\{a^2_1,a^2_2\}}{\max\{a^2_1,a^2_2\}}$ and $c_1=\frac{\max\{a^2_1,a^2_2\}}{\min\{a^2_1,a^2_2\}}$.

For $u,v\ge 0$ and $t\in [0,\frac{\pi}{2}]$,
\begin{align*}
\|\cos t\cdot u(x)+\sin t\cdot v(x)\|_{L^{2p}}&\ge \max(\cos t\|u\|_{L^{2p}},\sin t\|v\|_{L^{2p}})\\
&\ge \frac{\|u\|_{L^{2p}}\|v\|_{L^{2p}}}{\sqrt{\|u\|^2_{L^{2p}}+\|v\|^2_{L^{2p}}}}.
\end{align*}

Therefore, we have for some constant $C>0$ independent of $l,\theta$,
\begin{align}\label{3.26}
\int_{\mathbb{R}^N}\hat{h}^{2p}_1(l,\theta)dx\geq{Ce^{(p-1)Nsl}}~\text{and}~\int_{\mathbb{R}^N}\hat{h}^{2p}_2(l,\theta)dx\geq Ce^{(p-1)Nsl},
\end{align}
for all $(l,\theta)\in\mathbb{R}\times[0,\frac{\pi}{2}]$. Let $(u_1,v_1),(u_2,v_2)\in A^{+}$, and let $\hat{h}$ as the previous. From \eqref{3.25}, we can deduce there exists $l_0>0$ such that
\begin{align*}
\int_{\mathbb{R}^N}(|(-\Delta)^{\frac{s}{2}}\hat{h}_1(-l_0,\theta)|^2+|(-\Delta)^\frac{s}{2}\hat{h}_2(-l_0,\theta)|^2)dx\leq{K},
\end{align*}
for all $\theta\in[0,\frac{\pi}{2}]$, where $K$ is defined in Lemma \ref{l3.12}. For the choice of $l_0$, let
\begin{equation*}
\sigma_1(r):=\left\{
\begin{aligned}
&-r\star{(u_1,v_1)}=\hat{h}(-r,0),\qquad\qquad\qquad\quad\quad~~{0\leq r\leq{l_0}},\\
&\hat{h}(-{l_0},r-{l_0}),\qquad\qquad\qquad\qquad\quad\quad\quad\quad\quad~{{l_0}<r\leq{{l_0}+\frac{\pi}{2}}},\\
&(r-2{l_0}-\frac{\pi}{2})\star{(u_2,v_2)}=\hat{h}({r-2{l_0}-\frac{\pi}{2}},\frac{\pi}{2}),~{{l_0}+\frac{\pi}{2}<r\leq{2{l_0}+\frac{\pi}{2}}}.
\end{aligned}\right.
\end{equation*}
It is not difficult to check that $\sigma_1$ is a continuous path connecting $(u_1,v_1)$ and  $(u_2,v_2)$ and lying in $A^{+}$.
For the case that condition \eqref{3.24} is not satisfied. Suppose for instance
\begin{align*}
\int_{\mathbb{R}^N}(|(-\Delta)^\frac{s}{2}u_1|^2+|(-\Delta)^\frac{s}{2}v_1|^2)dx>\int_{\mathbb{R}^N}(|(-\Delta)^\frac{s}{2}u_2|^2+|(-\Delta)^\frac{s}{2}v_2|^2)dx.
\end{align*}
Then by Lemma \ref{l3.11}, there exists $l_1<0$ such that
\begin{align*}
\int_{\mathbb{R}^N}(|(-\Delta)^\frac{s}{2}(l_1\star u_1)|^2+|(-\Delta)^\frac{s}{2}(l_1\star v_1)|^2)dx=\int_{\mathbb{R}^N}(|(-\Delta)^\frac{s}{2}u_2|^2+|(-\Delta)^\frac{s}{2}v_2|^2)dx.
\end{align*}
Therefore, to connect $(u_1,v_1)$ and $(u_2,v_2)$ by a path in $A^{+}$, we can at first connect $(u_1,v_1)$ with $l_1\star(u_1,v_1)$ along arc $l*(u_1,v_1)$, then connect $l_1\star(u_1,v_1)$  with $(u_2,v_2)$.  This shows that $A^+$ is path connected. In a similar way, we can prove that $D^{+}$ is also path connected.
\end{proof}

From the previous notation,
\begin{align*}
F:&=\{(u,v)\in H_{a_1}\times H_{a_2}:G(u,v)=0\},
\end{align*}
define its radial subset and positive radial subset
\begin{align*}
F_{rad}:&=\{(u,v)\in H^{rad}_{a_1}\times H^{rad}_{a_2}:G(u,v)=0\},\\
F^+:&=\{(u,v)\in F: u\ge 0, v\ge 0\},\\
F^+_{rad}:&=\{(u,v)\in F_{rad}: u\ge 0,~v\ge 0\},
\end{align*}
where
\begin{align*}
G(u,v)=\int_{\mathbb{R}^N}(|(-\Delta)^\frac{s}{2}u|^2+|(-\Delta)^\frac{s}{2}v|^2)dx-\frac{(p-1)N}{2ps}\int_{\mathbb{R}^N}(\mu_1u^{2p}+2\beta u^pv^p+\mu_2v^{2p})dx.
\end{align*}

For $(u,v)\in H_{a_1}\times H_{a_2}$, let us set
\begin{equation*}
\Psi_{(u,v)}(l)=E(l\star{(u,v)}),
\end{equation*}
where $l\star{(u,v)}=(l\star u,l\star v)$ for short. Similar to the proof of Lemma \ref{l2.7}, we have

\begin{lemma}\label{l3.17}
For every $(u,v)\in H_{a_1}\times H_{a_2}$, there exists a unique $l_{(u,v)}\in{\mathbb{R}}$ such that $l_{(u,v)}\star{(u,v)}\in F$. Moreover $l_{(u,v)}$ is the unique critical point of $\Psi_{(u,v)}$, which is a strict maximum.
\end{lemma}

\begin{lemma}\label{l3.18}
There holds $\inf\limits_{F}E=\inf\limits_{F^+}E=\inf\limits_{F_{rad}^+}E$.
\end{lemma}
\begin{proof}
We prove the lemma by contradiction. Suppose there exists $(u,v)\in F$ such that
\begin{align}
0<E(u,v)<\inf\limits_{F^+}E.
\end{align}
For any $u\in{H^s(\mathbb{R}^N)}$, since $\|(-\Delta)^{\frac{s}{2}}|u|\|_{L^2}\le \|(-\Delta)^{\frac{s}{2}} u\|_{L^2}$, we get that $E(|u|,|v|)\le E(u,v)$ and $G(|u|,|v|)\le G(u,v)=0$. Thus, there exists $l_0\leq 0$ such that $G(l_0\star(|u|,|v|))=0$. We obtain that
\begin{align*}
E(l_0\star(|u|,|v|))&=(\frac12-\frac{s}{(p-1)N})e^{2s^2l_0}\bigg\{\int_{\mathbb{R}^N}(|(-\Delta)^\frac{s}{2}|u||^2+|(-\Delta)^\frac{s}{2}|v||^2)dx\bigg\}\\
&\leq(\frac12-\frac{s}{(p-1)N})e^{2s^2l_0}\bigg\{\int_{\mathbb{R}^N}(|(-\Delta)^\frac{s}{2}u|^2+|(-\Delta)^\frac{s}{2}v|^2)dx\bigg\}\\
&=e^{2s^2l_0}E(u,v).
\end{align*}

Therefore
$$
0<E(u,v)<\inf\limits_{F^+}E\leq E(l_0\star(|u|,|v|))\leq e^{2s^2l_0}E(u,v),
$$which contradicts $l_0\leq 0$. Thus $\inf\limits_{F}E=\inf\limits_{F^+}E$.

Next, if there exists  $(u,v)\in F^+$ such that
\begin{align*}
0<E(u,v)<\inf\limits_{F_{rad}^+}E.
\end{align*}
For $u\in{H^s(\mathbb{R}^N)}$, let $u^\ast$ denotes its Schwarz spherical rearrangement. According to the property of Schwarz symmetrization, we have that $E(u^\ast,v^\ast)\leq{E(u,v)}$ and $G(u^\ast,v^\ast)\leq{G(u,v)}=0$. Thus there exists $l_0\leq 0$ such that $G(l_0\star(u^\ast,v^\ast))=0$. Similarly, we get
\begin{align*}
0<E(u,v)<\inf\limits_{F_{rad}^+}E\leq E(l_0\star(u^\ast,v^\ast))\leq e^{2s^2l_0}E(u,v),
\end{align*}
which contradicts $l_0\leq 0$. Thus $\inf\limits_{F^+}E=\inf\limits_{F_{rad}^+}E$.
\end{proof}

\begin{proof}[{\bf Proof of Theorem \ref{t1.2}.}]
We have showed that $(\bar\lambda_1,\bar\lambda_2,\bar u,\bar v)$ is a solution of \eqref{GPE}. Since $(\bar u,\bar v)\in F^+_{rad}$, we just need to show that
\begin{align*}
E(\bar u,\bar v)=d\leq\inf\limits_{F^+_{rad}}E.
\end{align*}
Then $E(\bar u,\bar v)=\inf\limits_{F}E$ follows from  Lemma \ref{l3.18}. Choose any $(u,v)\in F^+_{rad}$. Let us consider the function $\Psi_{(u,v)}(l)=E(l*(u,v))$. By Lemma \ref{l3.11} there exists $l_0\gg1$ such that $(-l_0)\star(u,v)\in A^{+}$ and $l_0\star(u,v)\in D^{+}$. Therefore, the continuous path
\begin{equation*}
\bar\gamma(t)=((2t-1)l_0)\star(u,v)~~~~t\in[0,1],
\end{equation*}
connects $A^{+}$ with $D^{+}$, and by Lemma \ref{l3.15} and Lemma \ref{l3.17}, we can deduce that
\begin{align*}
d\leq\max_{t\in[0,1]}E(\bar\gamma(t))=E(u,v).
\end{align*}
Since this holds for all the elementary in $F^+_{rad}$, we  have
$$
d\leq\inf\limits_{F^+_{rad}}E.
$$

Finally, it remains to show that
\begin{equation}\label{f}
\inf_{F}E=\inf_{H_{a_1}\times H_{a_2}}\mathcal{R}.
\end{equation}
The proof of \eqref{f} is similar to the case for the single equation, see Lemma \ref{l2.6}.
\end{proof}

\section*{Acknowledgments}
This work was supported by the NSFC grant:11971184.



\begin{thebibliography}{99}


\bibitem{BC13} W. Bao \& Y. Cai, \textit{Mathematical theory and numerical method for Bose-Einstein condensation},  Kinet. Relat. Models, 6 (2013), 1-135.


\bibitem{BJ18} T. Bartsch \&  L. Jeanjean,  \textit{Normalized solutions for nonlinear Schr\"{o}dinger systems}, Proc. Roy. Soc. Edinburgh Sect. A, 148 (2018), 225-242.

\bibitem{BJN18} T. Bartsch, L. Jeanjean  \& N. Soave,  \textit{Normalized solutions for a system of coupled cubic Schr\"{o}dinger equations on $\mathbb{R}^3$},  J. Math. Pures Appl., 106 (2016), 583-614.

\bibitem{TZZ20} T. Bartsch, X. Zhong \& W. Zou, \textit{Normalized solutions for a coupled Schr\"{o}dinger system}, Math. Ann. (2020), doi: 10.1007/s00208-020-02000-w.

\bibitem{BM17} M. Bhakta  \& D. Mukherjee, \textit{Semilinear nonlocal elliptic equations with critical and supercritical exponents}, Commun. Pure Appl. Anal., 16 (2017), 1741-1766.

\bibitem{CS07} L. Caffarelli \& L. Silvestre, \textit{An extension problem related to the fractional Laplacian}, Comm. Partial Differential Equations, 32 (2007), 1245-1260.


\bibitem{FL2013} R. L. Frank \& E. Lenzmann, \textit{Uniqueness of non-linear ground states for fractional Laplacians in $\mathbb{R}$}, Acta Math., 210 (2013), 261-318.

\bibitem{FLS16} R. L. Frank, E. Lenzmann \& L. Silvestre, \textit{Uniqueness of radial solutions for the fractional Laplacian}, Comm. Pure Appl. Math., 69 (2016), 1671-1726.

\bibitem{G93} N. Ghoussoub, \textit{Duality and Perturbation Methods in Critical Point Theory},  Cambridge Tracts in Mathematics, 107, Cambridge, 1993, with appendices by David Robinson.


\bibitem{GJ18} T. Gou \& L. Jeanjean, \textit{Multiple positive normalized solutions for nonlinear Schr\"{o}dinger systems}, Nonlinearity, 31 (2018) 2319-2345.







\bibitem{I14} N. Ikoma, \textit{Compactness of minimizing sequences in nonlinear Schr\"{o}dinger systems under multiconstraint conditions}, Adv. Nonlinear Stud., 14 (2014), 115-136.

\bibitem{J97} L. Jeanjean, \textit{Existence of solutions with prescribed norm for semilinear elliptic equations}, Nonlinear Anal., 28 (1997), 1633-1659.

\bibitem{L1} N. Laskin, \textit{Fractional Schr\"{o}dinger equation}, Phys. Rev. E, 66 (2002), 056108.

\bibitem{L2} N. Laskin, \textit{Fractional quantum mechanics and L\'{e}vy path integrals}, Phys. Lett. A, 268 (2000), 298-305.

\bibitem{LC2018} C. Li, Z. Wu \& H. Xu, \textit{Maximum principles and B\^ocher type theorems}, Proc. Natl Acad. Sci. USA, 115 (2018), 6976-6979.

\bibitem{LYZ18} H. Li, Z. Yang \& W. Zou, \textit{On normalized solutions for nonlinear Schr\"{o}dinger equations}, Scientia Sinica Mathematica, doi: 10.1360/SSM-2020-0120.





\bibitem{RS14} X. Ros-Oton \& J. Serra, \textit{The Pohozaev identity for the fractional Laplacian}, Arch. Ration. Mech. Anal., 213 (2014), 587-628.

\bibitem{S07} L. Silvestre, \textit{Regularity of the obstacle problem for a fractional power of the Laplace operator}, Comm. Pure Appl. Math., 60 (2007), 67-112.

\bibitem{S20} N. Soave, \textit{Normalized ground states for the NLS equation with combined nonlinearities}, J. Differential Equation, 269 (2020), 6941-6987.

\bibitem{SN20} N. Soave, \textit{Normalized ground states for the NLS equation with combined nonlinearities: the Sobolev critical case},  J. Funct. Anal., 279 (2020), 108610.











\end{thebibliography}
\end{document}